\documentclass[reqno,english]{amsart}
\usepackage{amsfonts,amsmath,latexsym,verbatim,amscd,mathrsfs,color,array}
\usepackage[colorlinks=true]{hyperref}

\usepackage{amsmath,amssymb,amsthm,amsfonts,graphicx,color}
\usepackage{amssymb}
\usepackage{pdfsync}
\usepackage{epstopdf}
\usepackage{amsmath,latexsym,amsthm}
\usepackage{colonequals}
\usepackage{multicol}

\numberwithin{equation}{section}
\def\R{\mathbb{R}}
\def\N{\mathbb{N}}

\makeatletter
\def\@currentlabel{2.1}\label{e:dispaa}
\def\@currentlabel{2.21}\label{e:dispau}
\def\@currentlabel{2.22}\label{e:dispav}
\def\@currentlabel{2.23}\label{e:dispaw}
\def\@currentlabel{2.24}\label{e:dispax}
\def\theequation{\thesection.\@arabic\c@equation}
\makeatother

\newcommand{\D}{\displaystyle}

\newcommand{\inn}{{\quad\hbox{in } }}
\newcommand{\onn}{{\quad\hbox{on } }}
\newcommand{\ttt}{\tilde }

\newcommand{\pp}{ {\partial} }

\newcommand{\cuad}{{\sqcap\kern-.68em\sqcup}}

\newcommand{\DD}{{\mathcal D}}

\newcommand{\be}{\begin{equation}}
\newcommand{\ee}{\end{equation}}

\newcommand{\la}{\lambda}

\newcommand{\eq}[1]{\begin{equation}#1\end{equation}}
\newcommand{\equ}[1]{\begin{equation*}#1\end{equation*}}

\newcommand{\ddn}[1]{\frac{\partial #1}{\partial \nu}}
\renewcommand{\theequation}{\thesection.\arabic{equation}}
 \newtheorem{lema}{Lemma}[section]
 \newtheorem{lemma}{Lemma}[section]

\newtheorem{teo}{Theorem}

\newtheorem{prop}{Proposition}[section]
\newtheorem{corollary}{Corollary}[section]
\newtheorem{remark}{Remark}[section]
\newcommand{\bremark}{\begin{remark} \em}
\newcommand{\eremark}{\end{remark} }

\begin{document}

\title[Interior bubbling solutions for the critical Lin-Ni-Takagi problem]{Interior bubbling solutions for the critical Lin-Ni-Takagi problem in dimension 3}

\author[M. del Pino]{Manuel del Pino}
\address{\noindent   Departamento de
Ingenier\'{\i}a  Matem\'atica and Centro de Modelamiento
 Matem\'atico (UMI 2807 CNRS), Universidad de Chile, Santiago,
Chile.}
\email{delpino@dim.uchile.cl}

\author[M. Musso]{Monica Musso}
\address{\noindent   Departmento de Matem\'aticas, Pontificia Universidad Cat\'olica de Chile, Santiago, Chile.}
\email{mmusso@mat.puc.cl }

\author[C. Rom\'an]{Carlos Rom\'an}
\address{\noindent   Sorbonne Universit\'es, UPMC Univ Paris 06, CNRS, UMR 7598, Laboratoire Jacques-Louis Lions, 4, place Jussieu 75005, Paris, France.}
\email{roman@ann.jussieu.fr }

\author[J. Wei]{Juncheng Wei}
\address{\noindent  Department of Mathematics, University of British Columbia, Vancouver, BC V6T1Z2, Canada}
\email{jcwei@math.ubc.ca}

%\thanks{The research of the second  author
%has been partly supported by Fondecyt Grant 1120151 and Millennium Nucleus Center for Analysis of PDE, NC130017.}

\begin{abstract} We consider the problem of finding positive solutions of the problem $\Delta u - \lambda u +u^5 = 0$ in a bounded, smooth domain $\Omega$ in $\R^3$, under zero Neumann boundary conditions. Here  $\lambda$ is a positive number. We analyze the role of Green's function of $-\Delta +\lambda$ in the presence of solutions exhibiting single bubbling behavior at one point of the domain when $\lambda$ is regarded as a parameter. As a special case of our results, we find and characterize a positive value $\la_*$  such  that if $\lambda-\lambda^*>0$ is sufficiently small, then this problem is solvable by a solution $u_\la$ which blows-up by {\em bubbling} at a certain interior point of $\Omega$ as $\la \downarrow \la_*$.
\end{abstract}
\maketitle

\section{Introduction}

Let $\Omega$ be a bounded, smooth domain in $\R^n$.
This paper deals with the boundary value problem
\be
\Delta u - \la u + u^p = 0 \inn \Omega, \nonumber
\ee
\be
u>0  \inn   \Omega, \label{1}
\ee
\be
\frac {\pp u}{\pp \nu} = 0 \onn \pp\Omega.
\nonumber \ee
where $p>1$.  A large literature has been devoted to this problem when $1\le p\le \frac{n+2}{n-2}$ for asymptotic values of the parameter $\la$.
A very interesting feature of this problem is the presence of families of solutions $u_\la$ with {\em point concentration phenomena}. This means solutions that exhibit
peaks of concentration around one or more points of $\Omega$ or $\pp\Omega$, while being very small elsewhere.
 For $1<p<\frac{n+2}{n-2}$, solutions with this feature around
 points of the boundary where first discovered by
Lin, Ni and Takagi in \cite{LNT} as $\la \to + \infty$.
It is found in \cite{LNT,NT1,NT2} that a mountain pass, or least energy positive solution $u_\la $ of Problem \eqref{1} for  $\la\to +\infty $ must look like
$$u_\la (x) \sim \la^{\frac 1{p-1}} V(\la^{\frac 12 }( x- x_\la)) $$
where $V$ is the unique positive radial solution to
\be \label{3}
\Delta V - V + V^p = 0, \quad \lim_{|y|\to \infty} V(y) = 0   \inn \R^N
\ee
and $x_\la \in \pp \Omega$ approaches a point of maximum mean curvature of $\pp\Omega$. See \cite{df1} for a short proof of this fact. Higher energy solutions with this asymptotic profile near one or several points of the boundary or the interior of $\Omega$ have been constructed and analyzed in many works, see for instance  \cite{DY,DFW,GPW,GW,LNW} and their references. In particular, solutions with any given number of interior and boundary concentration points are known to exist as $\la \to +\infty$.

\medskip
The case of the critical exponent $p=\frac{n+2}{n-2}$ is in fact quite different. In particular, no positive solutions of \eqref{3} exist. In that situation solutions $u_\la$ of \eqref{1} do exist
for sufficiently large values of $\la$ with concentration now in the form
\be u_\la  (x) \sim   \mu_\la ^{-\frac {n-2} 2 }  U \left (\mu_\la^{-1 }( x- x_\la)\right ) \label{5}\ee
where $\mu_\la = o( \la^{-\frac 12} ) \to 0 $ as $\la \to + \infty$. Here
$$ U(x) =  \alpha_n \left ( \frac  1 { 1+ |y|^2} \right )^{\frac {n-2}2},  \quad \alpha_n = (n(n-2))^{\frac {n-2} 4} , $$
 is the {\em standard bubble}, which up to scalings and translations, is the unique positive solution of the Yamabe equation
 \be \label{6}
\Delta U +  U^{\frac{n+2}{n-2}} = 0 \inn \R^N.
\ee
Solutions with boundary bubbling have been built and their dimension-dependent bubbling rates $\mu_\la$ analyzed in various works, see
\cite{AM,APY,DRW,GG,GL,NPT,WWY1,ZQW,WY} and references therein. Boundary bubbling by small perturbations of the exponent $p$ above and below the critical exponent has been found in \cite{DMP}.

\medskip
Unlike the subcritical range, for $p= \frac{n+2}{n-2}$  solutions  with interior bubbling points as $\la \to +\infty$ are harder to be found. They do not exist for $n=3$ or $n\ge 7$,  \cite{E,R1,R2}, and in all dimensions interior bubbling can only coexist with boundary bubbling \cite{R2}.  To be noticed is that the constant function $ \underline u_\la :=  \la^{\frac 1{p-1}}$ represents a trivial solution to Problem \eqref{1}. A compactness argument yields that this constant is the unique solution of \eqref{1} for $1<p<\frac{n+2}{n-2}$  for all sufficiently small $\la$, see \cite{LNT}.
The  { \em Lin-Ni conjecture}, raised in \cite{LN}  is that this is also the case for $p=\frac{n+2}{n-2}$. The issue turns out to be quite subtle. In \cite{AY1,AY2} it is found that
radial nontrivial solutions for all small $\la>0$ exist when $\Omega$ is a ball in dimensions $n=4,5,6$, while no radial solutions are present for small $\la$  for $n=3$ or $n\ge 7$.   For a general convex domain, the Lin-Ni conjecture is true in dimension $n=3$ \cite{WX,Z}. See \cite{DRW} for the extension to the mean convex case and related references.
In \cite{RW} solutions with multiple  interior bubbling points  when $\la\to 0^+$ were found when $n=5$, in particular showing that Lin-Ni's conjecture fails in arbitrary domains in this dimension. This result is the only example present in the literature of its type. The authors conjecture that a similar result should be present for $n=4,6$.

\medskip
In the case $n=3$ interior bubbling is not possible if $\la \to +\infty$ or if $\la \to 0^+$, for instance in a convex domain. In this paper we will show a new phenomenon, which  is
the presence of a solution $u_\la$ with {\em interior bubbling}  for values of $\la$ near a  number $0< \la_*(\Omega) < +\infty$ which can be explicitly characterized.
Thus, in what remains of this paper we consider the critical problem

\be
\Delta u - \la u + u^5 = 0 \inn \Omega, \nonumber
\ee
\be
u>0  \inn   \Omega, \label{7}
\ee
\be
\frac {\pp u}{\pp \nu} = 0 \onn \pp\Omega.
\nonumber \ee
 where $\Omega\subset \R^3$ is smooth and bounded.  We will prove the following result.

\begin{teo}\label{teo1}
There exists a number $0<\la_*<+\infty $ such that for all $\la > \la^*$ with  $\la -\la_*$ sufficiently small, a nontrivial solution $u_\la$ of Problem \eqref{7} exists, with an
asymptotic profile as $\la \to \la_*^+$ of the form
$$
u_\la (x) =  3^{\frac 14} \left ( \frac  {\mu_\la}  { \mu_\la^2 + |x-x_\la |^2} \right )^{\frac {1}2} + O( \mu_\la^{\frac 12} ) \inn \Omega,
$$
where $\mu_\la  = O (\la - \la_*) $ and the point $x_\la\in \Omega$ stays uniformly away from $\pp \Omega$.

\end{teo}
The number $\la_*$ and the asymptotic location  of the points $x_\la$ can be characterized in the following way.
For $\la >0$ we let $G_\la (x,y) $ be the Green function of the problem
$$
\Delta_x G_\la (x,y) - \la G_\la (x,y) + \delta_y(x) =0 \inn \Omega
$$
$$
\frac {\pp G_\la }{\pp \nu} (x,y) = 0 \onn \pp \Omega
$$
so that, by definition
\be \label{Glambdadef}
G_\la (x,y) = \Gamma(x,y) - H_\la (x,y)
\ee
where $\Gamma(x,y)=\frac {1} {4\pi|x-y|}$ and $H_\la$, the regular part of $G_\la$, satisfies
\be\label{robin}
\Delta_x H_\la (x,y) - \la H_\la (x,y) -  \frac {1 } {4\pi|x-y|}   =0 \inn \Omega
\ee
$$
\frac {\pp H_\la }{\pp \nu} (x,y) =  \frac {\pp}{\pp \nu} \frac {1} {4\pi|x-y|}   \onn \pp \Omega
$$
Let us consider the {\em diagonal} of the regular part (or Robin's function)
\be
g_\la (x) :=  H_\la (x,x), \quad x\in \Omega .
\label{gla}\ee
Then we have (see Lemma \ref{lemag}) $g_\la (x) \to -\infty$ as $x\to \pp \Omega$. %Besides, for each fixed $x\in \Omega$, we have that $$
%\lim_{\la \to +\infty } g_\la (x) = -\infty $$
The number $\la_*(\Omega) $ in Theorem \ref{teo1} is characterized as
\be
\la_*(\Omega)  := \inf \{ \la >0\ /\   \sup_{x\in \Omega}  g_\la (x) < 0 \} .
\label{lastar}\ee
In addition, we have that the points $x_\la\in\Omega  $ are such that
\be
\lim_{\la \downarrow \la_*}  g_\la (x_\la) =  \sup_{ \Omega}  g_{\la_*}   = 0.
\label{limite} \ee
As we will see in \S 2, in the ball $\Omega = B(0,1)$, the number $\la_*$ is the unique number $\la$ such that
$$
\frac{\sqrt\la-1}{\sqrt \la +1}\exp{(2\sqrt \la)}=1,
$$
so that  $\la^*\approx 1.43923$.

\bigskip
It is worthwhile to emphasize the connection between the number $\la_*$ and the so called {\em Brezis-Nirenberg number}  $\ttt \la^*(\Omega)>0 $ given as the least value $\la$ such that for all $  \ttt \la^* < \la < \la_1$ where $\la_1$ is the first Dirichlet eigenvalue of the Laplacian,   there exists a least energy  solution of the $3d$-Brezis-Nirenberg problem \cite{bn}

\be
\Delta u + \la u + u^5 = 0 \inn \Omega, \nonumber
\ee
\be
u>0  \inn  \ \  \Omega, \label{ee}
\ee
\be
u = 0  \onn \pp\Omega.
\nonumber \ee
A parallel characterization of the number $\ttt \la_*$ in terms of a Dirichlet Green's function has been established in \cite{druet} and its role in bubbling phenomena further explored in \cite{DDM}.  It is important to remark that the topological nature of the solution we find is not that of a least energy, mountain pass type solution (which is actually just the constant for small $\la$). In fact the construction formally yields that its Morse index is 4.

\medskip
 Our result can be depicted (also formally) in Figure \ref{fig} as a bifurcation diagram from the branch of constant solutions $u=\underline u_\la $. At least in the radial case, what our result suggests is that the bifurcation branch which stems from the trivial solutions at the value $\la = \la_2/4$, where $\la_2$ is the first nonzero radial eigenvalue of $-\Delta$ under zero Neumann boundary conditions in the unit ball, goes left and ends at $\la = \la_*$. In dimensions $n=4,5,6$ the branch ends at $\la =0$ while for $n\ge 7$ it blows up to the right.

\begin{figure}[h]
\label{fig}
\includegraphics[scale=0.42]{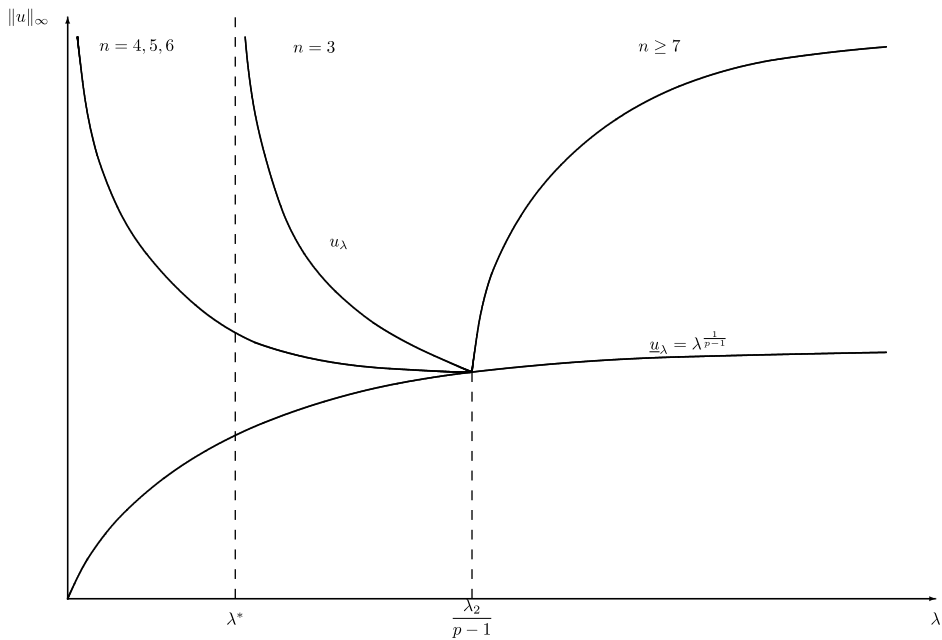}
\caption{Bifurcation diagram for solutions of Problem \ref{1},  $p=\frac{n+2}{n-2}$}
\end{figure}

\medskip

\medskip
Theorem \ref{teo1}, with the additional properties
 will be found a consequence of a more general result, Theorem \ref{teo2} below,  which concerns critical points with value zero for the function $g_{\la_0}$ at a value $\la_0 >0$.
 We state this result and find Theorem \ref{teo1} as a corollary in \S 2, as a consequence of general properties of the functions $g_\la$.  The remaining sections will be devoted to the proof
 of Theorem \ref{teo2}.

\section{ Properties of $g_\la$ and statement of the main result }

Let $g_\la(x)$ be the function defined in \eqref{gla}. Our main result states that an interior bubbling solution is present as $\la\downarrow \la_0$, whenever $g_{\la_0}$ has either a local minimum or a nondegenerate critical point with value $0$.

\begin{teo} \label{teo2} Let us assume that for a number $\lambda_0> 0$, one of the following two situations holds.

\noindent{$(a)$} There is an open subset $\DD$ of $\Omega$ such that
\be\label{pico}
0 = \sup_{\DD} g_{\lambda_0} > \sup_{\partial \DD} g_{\lambda_0}\,.
\ee
\noindent{$(b)$} There is a point $x_0\in \Omega$ such that
$g_{\lambda_0} (x_0) = 0, \quad \nabla g_{\lambda_0} (x_0) = 0$
and $D^2_xg_{\lambda_0} (x_0)$ is non-singular.

\smallskip \noindent Then for all $\lambda >\lambda_0$ sufficiently close to $\lambda_0$ there exists a solution $u_\lambda$ of Problem $(\ref{1})$ of the form
\begin{equation}u_\lambda (x) \, = \, 3^{1/4}\,\left (  \frac { \mu_\lambda}  { \mu_\la^2  \,+\, |x-  x_\lambda |^2 }  \right )^{\frac 12} + O ( \mu_\la^{\frac 12} ) ,\quad
\mu_\la =  \gamma \frac{g_\la(x_\la)}\la >0
\end{equation}
for some $\gamma >0$.
Here $x_\lambda \in \DD$ in case $(a)$ holds and  $x_\lambda \to x_0$ if $(b)$ holds.
Besides, for certain positive numbers $\alpha,\beta $ we have that
\be  \alpha (\la- \la_0) \le g_\la(x_\la)  \le \beta (\la- \la_0). \label{pico1}\ee
\end{teo}
%---------------------------------------------------------------------

Of course, a natural question is whether or not values $\la_0$ with the above characteristic do exist. We shall prove that the number $\la_*$ given by \eqref{lastar} is indeed positive and finite, and that $\la_0= \la_*$ satisfies \eqref{pico}. That indeed proves Theorem \ref{teo1} as a corollary of Theorem \ref{teo2}.

\medskip
Implicit in condition (b) is the fact that $g_{\la_0}(x)$ is a smooth function and in inequality \eqref{pico1} the fact that $g_\la$ increases with $\la$. We have the validity of
the following result.

In this section we establish some properties of the function $g_\la(x)$ defined in \eqref{gla}.
We begin by proving that $g_\lambda (x)$ is a smooth function, which is strictly increasing in $\la$.

\begin{lemma}\label{primerlema}
The function $g_\lambda $ is of class $C^{\infty} (\Omega )$. Furthermore, the function ${\partial g_\lambda \over \partial \lambda }$ is well defined, smooth and strictly positive in $\Omega$. Its derivatives depend continuously on $\lambda$.
\end{lemma}
%---------------------------------------------------------------------
\begin{proof}We will show that $g_\lambda \in C^k $, for any $k$. Fix $x \in \Omega $. Let $h_{1,\lambda }$ be the function defined in $\Omega \times \Omega$ by the relation
\equ{
H_\lambda (x,y) = \beta_1\, |x-y| + h_{1,\lambda } (x,y)\,,}
where $\beta_1 =-{\lambda \over 8 \pi}$. Then $h_{1,\lambda }$ satisfies the boundary value problem
\equ{\left\lbrace
\begin{array}{rcll}
-\Delta_y h_{1,\lambda} +\lambda h_{1,\lambda}&=&-\lambda \beta_1|x-y|& \mbox{in } \Omega,\\
\ddn{h_{1,\lambda}(x,y)}&=&\ddn{\Gamma(x-y)}-\beta_1\ddn{|x-y|}&\mbox{on } \partial \Omega.
\end{array}\right.}

Elliptic regularity then yields that $h_{1, \lambda }(x,\cdot) \in C^{2} (\Omega)$. Its derivatives are clearly continuous as functions of the joint variable. Let us observe that the function $H_\lambda(x,y) $ is symmetric, thus so is $h_1$, and then $h_{1,\lambda}(\cdot,y)$ is also of class $C^2$ with derivatives jointly continuous. It follows that $h_{1 , \la} (x,y)$ is a function of class $C^2(\Omega\times\Omega)$. Iterating this procedure, we get that, for any $k$
$$
\label{acca}
H_\lambda (x,y) = \sum_{j=1}^k \beta_j |x-y|^{2j-1} + h_{k,\lambda} (x,y)
$$
with $\beta_{j+1} = -\lambda\,\beta_j/((2j+1)(2j+2))$ and $h_{k,\lambda}$ solution of the boundary value problem
\equ{\left\lbrace
\begin{array}{rcll}
-\Delta_y h_{k,\lambda} +\lambda h_{k,\lambda}&=&-\lambda \beta_k|x-y|^{2k-1}& \mbox{in }\Omega,\\
\ddn{h_{k,\lambda}(x,y)}&=&\ddn{\Gamma(x-y)}-\sum_{j=1}^k\beta_j\ddn{|x-y|^{2j-1}}&\mbox{on } \partial \Omega.
\end{array}\right.}
We may remark that
$-\Delta_yh_{k+1,\lambda}+\lambda\,h_{k,\lambda}=0$ in $\Omega\,.$
Elliptic regularity then yields that $h_{k,\lambda }$, is a function of class $ C^{k+1} (\Omega \times\Omega)$. Let us observe now that by definition of $g_\lambda$ we have $ g_\lambda (x) = h_{k, \lambda } (x,x), $ and this concludes the proof of the first part of the lemma.

\medskip
For a fixed given $x\in \Omega$, consider now the unique solution $F(y)$ of
$$
-\Delta_y F + \lambda\, F=G(x,y) \quad  y\in \Omega,\quad
\ddn{F}=0 \quad y\in \partial \Omega.
$$
Elliptic regularity yields that $F$ is at least of class $C^{0,\alpha}$. A convergence argument shows that actually
$F(y) = {\partial H_\lambda \over \partial \lambda } (x,y)\,.$
Since $\lambda>0$ and $G$ is positive in $\Omega$, using $F_-$ as a test function we get that $F_-=0$ in $\Omega$, thus $F>0$. Hence, in particular
$
{\partial g_\lambda \over \partial \lambda} (x) = F(x) >0\,.$
Arguing as before, this function turns out to be smooth in~$x$. The resulting expansions easily provide the continuous dependence in $\lambda$ of its derivatives in the $x$-variable.

\end{proof}

%\smallskip
%\noindent

\begin{lemma}\label{lemag}
For each fixed $\la>0$ we have that
\begin{equation}
\label{terra1}
g_\la (x) \to -\infty, \quad {\mbox {as}} \quad x\to \partial \Omega.
\end{equation}
We define
$$
M_\la = \sup_{x \in \Omega} g_\la (x).
$$
Then
\begin{equation}\label{uno}
M_\la \to - \infty \quad {\mbox {as }  } \la \to 0^+,
\end{equation}
and
\begin{equation}\label{unouno}
M_\la > 0   \quad {\mbox {as }  }  \quad \la \to +\infty.
\end{equation}

\end{lemma}

\medskip
\noindent

\begin{proof}

We prove first \eqref{terra1}.
Let $x \in \Omega$ be such that $d:= {\mbox {dist}} (x, \partial \Omega)$ is small. Then there exists a unique $\bar x \in \partial \Omega$ so that $d= |x-\bar x|$. It is not restrictive to assume that $\bar x = 0 $ and that the outer normal at $\bar x$ to $\partial \Omega$ points toward the $x_3$-direction. Let $x^*$ be the reflexion point, namely $x^* = (0,0,-d)$ and consider $H^* (y,x) = {1\over 4\pi | y-x^*|}$.  The function $y \to H^* (y,x)$ solves
$$
-\Delta_y \phi + \la \phi = \la \Gamma (y-x^* ) , \quad y \in \Omega, \quad {\partial \phi \over \partial \nu} =  {\partial \Gamma \over \partial \nu} (y-x^* ) , \quad y \in \partial \Omega.
$$
Observe now that $$\Gamma (y-x^* ) = {1\over 4\pi |x-y|} + {1\over 4\pi}  \left[ {|y-x|-|y-x^*| \over |y-x | \, |y-x^*| } \right] = {1\over 4\pi |x-y|}  + O(1), $$
 with $O(1)$ uniformly bounded, as $d \to 0$, for $y \in \partial \Omega$. This gives that $H_\la (y,x) = - H^* (y,x) + O(1)$, as $d \to 0$. Thus
$$
H_\la (x,x) = - {1\over 4\pi {\mbox {dist}} (x, \partial \Omega)} + O(1),
$$
as $d\to 0$. So we conclude the validity of \eqref{terra1}.

\medskip
Next we prove \eqref{uno} and \eqref{unouno}.

\noindent
Proof of \eqref{uno}. \ \ Let $p(x) := {1\over |\Omega| } \int_{\Omega} H_\la (x,y ) \, dy$. Observe that
$$
p(x) =  {1\over |\Omega| } \int_{\Omega} \Gamma (x-y ) \, dy +  {1\over \la |\Omega| } \int_{\Omega} \Delta H_\la (x,y ) \, dy = $$  $$
{1\over |\Omega| } \int_{\Omega} \Gamma (x-y ) \, dy +  {1\over \la |\Omega| } \int_{\partial \Omega} {\partial H_\la  \over \partial \nu }  \, d\sigma (y)=
$$
$$
{1\over |\Omega| } \int_{\Omega} \Gamma (x-y ) \, dy +  {1\over \la |\Omega| }
\int_{\partial \Omega} {\partial \Gamma  \over \partial \nu } (x-y)  \, d\sigma (y)
= -{a \over \la |\Omega | } + p_0 (x)
$$
where $a$ is a positive constant and $p_0 (x)$ is a bounded function. Define now
$H_0 (x,y)$ to be the bounded solution to
$$
-\Delta H_0 = {a\over |\Omega |} , \quad {\partial H_0 \over \partial \nu} = {\partial \Gamma \over \partial \nu } (x-y) \quad y \in \partial \Omega, \quad \int_{\Omega} H_0 = 0.
$$
We write
\begin{equation}
\label{terra2}
H_\la (x,y) = \underbrace{-{a\over \la |\Omega |} + p_0 (x) }_{=p(x) } + H_0 (x,y) + \hat H (x,y).
\end{equation}
By definition, $\hat H$ solves
$$
-\Delta \hat H + \la \hat H= \la \left[ \Gamma (x-y) - H_0 (x,y) + p_0 (x) \right], \quad
{\partial \hat H \over \partial \nu}=0 \quad {\mbox {on}} \quad \partial \Omega, \quad \int_{\Omega} \hat H = 0.
$$
Thus we have that $\hat H = O(1) $, as $\la \to 0$. Taking this into account, from decomposition \eqref{terra2} we conclude that
$$
\max_{x \in \Omega } g_\la (x) := \max_{x \in \Omega} H_\la (x,x) \leq -{a \over \la |\Omega|} + O(1) \to -\infty, \quad {\mbox {as}} \quad \la \to 0.
$$
This proves \eqref{uno}.

\medskip
\noindent
Proof of \eqref{unouno}. \ \  Assume, by contradiction, that for some sequence $\la_n \to \infty$, as $n \to \infty$, one has $\max_{x\in \Omega} g_{\la_n} (x) \leq -{1\over n}$. Fix $x_0 \in \Omega$, so that dist $(x_0 , \partial \Omega ) = \max_{x \in \Omega} {\mbox {dist}} (x, \partial \Omega )$. Thus we have that $-\Delta_y H_{\la_n} (y, x_0 ) \to \infty$, as $n \to \infty$. But on the other hand, a direct application of divergence theorem gives
$$
\int_\Omega \left( - \Delta_y H_{\la_n } (y, x_0 ) \right) \, dy = - \int_{\partial \Omega}
{\partial \Gamma \over \partial \nu} (x_0 - y) d\sigma (y) .
$$
The left side of the above identity converges to $\infty$ as $n \to \infty$, while the right and side is bounded. Thus we reach a contradiction, and \eqref{unouno} is proved.

%\medskip

\end{proof}

\medskip
The above considerations yield Theorem \ref{teo1} as a consequence of Theorem \ref{teo2}.

\medskip

\noindent
\begin{corollary}
The number $\la_*$ given by $\eqref{lastar}$ is well-defined and $0<\la_* < +\infty$. Besides,
the statement of Theorem \ref{teo1} holds true.
\end{corollary}

\proof
From Lemma \ref{primerlema}, and relations \eqref{uno} and \eqref{unouno}, we deduce that the number $\la_*$ is finite and positive.
Besides, by its definition and the continuity of $g_\la$, it clearly follows that
$$
 \sup_{x\in \Omega} g_{\la_*} (x) = 0.
$$ and that
there is an open set $\DD$ with compact closure inside $\Omega$ such that
$$
\sup_{\pp \DD} g_{\la_*}  < \sup_{\DD} g_{\la_*} = 0   .
$$
Hence,   Theorem \ref{teo1} follows from Theorem \ref{teo2}.
\qed
%--------------------------------------------------------------------

\bigskip
As it was stated in the introduction,
the number  $\lambda^*(\Omega)$ can be explicitly computed in the case $\Omega=B(0,1)$ as the following Lemma shows.

\begin{lemma}

Let $\Omega=B(0,1)$.
The number $\la^*$ defined in \eqref{lastar} is the unique solution of the equation
$$
\frac{\sqrt\la-1}{\sqrt \la +1}\exp{(2\sqrt \la)}=1,
$$
so that  $\la^*\approx 1.43923$.

\end{lemma}

\proof
The minimizer of $H_\la(x,x)$ is attained at $x=0$.
We compute the value $H_\lambda(0,0)$ for $\lambda>0$. The function $G_\lambda(0,y)$ is radially symmetric and it satisfies the equation

\begin{equation}
\label{gr}
-\Delta_y G_\lambda + \lambda\, G_\lambda =  \delta_0 \quad y\in B(0,1), \quad
\pp_r G_\lambda (0,y) = 0 \quad \; y \in \partial B(0,1).
\end{equation}
Letting $r=|y|$, we have
\eq{G_\la(0,y)=\frac{1}{4\pi r}
\left[e^{-\sqrt \la r}+\frac{2\sinh(\sqrt \la r)}{1+\frac{\sqrt\la-1}{\sqrt \la +1}\exp{(2\sqrt \la)}}
\right]\label{GreenBall}.}
Indeed, $ {e^{\sqrt{\la} r } \over r}$ and $ {e^{-\sqrt{\la} r } \over r}$ are radial solutions to $\Delta \phi +\la \phi=0$ for $|y|>0$. If we define
$$
\mathcal{G}_A (r) =\frac{A}{r} \left[e^{\sqrt \la r}+e^{2\sqrt \la }\frac{\sqrt \la -1}{\sqrt \la +1}e^{-\sqrt \la r} \right],
$$
where $A$ is a constant, then
 $\pp_r\mathcal{G}_A =0$ on  $\partial B(0,1)$.
Since $\lim_{|y|\rightarrow 0}|y|G_\la (0,y)=\frac{1}{4\pi}$,
if we choose
$$
A_\la=\frac{1}{4\pi}\frac{1}{1+\frac{\sqrt \la-1}{\sqrt \la+1}\exp{(2\sqrt \la)}}
$$
then $\mathcal{G}_{A_\la}$ satisfies \eqref{gr}. By uniqueness $\mathcal{G}_{A_\la}=G_\la (0,y)$, and we get \eqref{GreenBall}. Thus
$$
H_\la(0,y)=\frac{1}{4\pi r}
\left[(1-e^{-\sqrt \la r})-\frac{2\sinh(\sqrt \la r)}{1+\frac{\sqrt\la-1}{\sqrt \la +1}\exp{(2\sqrt \la)}}
\right],
$$
and
$$
g_\la(0)=H_\la(0,0)=\frac{1}{4\pi}\left[\sqrt \la - \frac{2\sqrt \la }{1+\frac{\sqrt \la -1}{\sqrt \la +1}\exp{2\sqrt \la}} \right].
$$
We deduce that $\la^*$ is the unique value such that $g_{\la^*} (0) =0$, therefore $\la^*$ satisfies
$$
\frac{\sqrt\la-1}{\sqrt \la +1}\exp{(2\sqrt \la)}=1.
$$
Then $\la^*\approx 1.43923$.
\qed

\bigskip

  The rest of this paper will be devoted to the proof of Theorem \ref{teo2}.  The proof consists of the construction of a suitable first approximation of a bubbling solution
  for $\la$ slightly above $\la_0$. The problem is then reduced, via a finite dimensional variational reduction procedure, to one in which the variables are the location of the bubbling point and the corresponding scaling parameter. That problem can be solved thanks to the assumptions made. We carry out this procedure in Sections \S 3-\S 7.

The rest of this work will be devoted to the proof of Theorem \ref{teo2}.
In Section \ref{sec2} we define an approximate solution $U_{\zeta , \mu}$, for any given point $\zeta \in \Omega$, and any positive number $\mu$,
and we compute its energy $E_\lambda ( U_{\zeta , \mu})$, where
\eq{\label{Energy}
E_\lambda(u)=\frac{1}{2}\int_\Omega |\nabla u|^2+\frac{\lambda}{2}\int_\Omega |u|^2-\frac{1}{6}\int_\Omega |u|^6.}
In Section \ref{Sec:Critical1} we establish that in the situation of Theorem \ref{teo2} there are critical points of $E_{\lambda}(U_{\mu,\zeta})$ which persist under properly small perturbations of the functional.
Observe now that,
for $\varepsilon>0$, if we consider the transformation
\equ{u(x)=\frac{1}{\varepsilon^{1/2}}v\left(\frac{x}{\varepsilon}\right)}
then $v$ solves the problem
\eq{\left\lbrace
\begin{array}{rcll}
-\Delta v+\varepsilon^2\lambda v-v^5&=&0&
v>0\quad \mbox{in } \quad \Omega_\varepsilon,\\
\ddn{v}&=&0&\mbox{on }\partial \Omega_\varepsilon,
\end{array}\right.\label{problem2}}
where $\Omega_\varepsilon=\varepsilon^{-1}\Omega$. We will look for a solution of (\ref{problem2}) of the form $v=V+\phi$,
where $V$ is defined as
$U_{\zeta,\mu}(x)=\frac{1}{\varepsilon^{1/2}}V\left(\frac{x}{\varepsilon}\right)$,
and $\phi $ is a smaller perturtation.
In Section \ref{sec4} we discuss a linear problem that will be useful to find the perturbation $\phi$. This is done in Section \ref{sec5}. We conclude our construction in the final argument, in Section \ref{Sec:Variational}.

\section{Energy expansion}\label{sec2}

We fix a point $\zeta\in\Omega$ and a positive number $\mu$. We denote in what follows
\equ{w_{\zeta,\mu}(x)=3^{1/4}\frac{\mu^{1/2}}{\sqrt{\mu^2+|x-\zeta|^2}}}
which correspond to all positive solutions of the problem
\equ{-\Delta w-w^5=0,\quad \mbox{in } \R^3.}
We define $\pi_{\zeta,\mu}(x)$ to be the unique solution of the problem
\eq{\label{problempi}\left\lbrace
\begin{array}{rcll}
-\Delta \pi_{\zeta,\mu}+\lambda \pi_{\zeta,\mu}&=&-\lambda w_{\zeta,\mu}&\mbox{in }\Omega,\\
\ddn{\pi_{\zeta,\mu}}&=&-\ddn{w_{\zeta,\mu}}&\mbox{on }\partial \Omega.
\end{array}\right.}
We consider as a first approximation of the solution of (\ref{1}) one of the form
\begin{equation}\label{ansatz1}U_{\zeta,\mu}=w_{\zeta,\mu}+\pi_{\zeta,\mu}.
\end{equation}
Observe that $U_{\zeta,\mu}$ satisfies the problem
\eq{\left\lbrace
\begin{array}{rcll}
-\Delta U_{\zeta,\mu}+\lambda U_{\zeta,\mu}&=&w_{\zeta,\mu}^5&\mbox{in }\Omega,\\
\ddn{U_{\zeta,\mu}}&=&0&\mbox{on }\partial \Omega.
\end{array}\right.\label{ProbAprox}}
Let us also observe that
\equ{\int_{\Omega}w^5_{\zeta,\mu}=C\mu^{1/2}\left(1+o(1)\right),\quad \mbox{as }\mu\rightarrow 0,}
which implies that
$ \frac{w^5_{\zeta,\mu}}{\int_{\Omega}w^5_{\zeta,\mu}}\rightarrow 0,$ as $\mu\rightarrow 0 $,
uniformly on compacts subsets of $\overline{\Omega}\setminus\lbrace\zeta\rbrace$. It follows that on each of this subsets
\eq{U_{\zeta,\mu}(x)=\left(\int_{\Omega}w_{\zeta,\mu}^5\right)G(x,\zeta)=C\mu^{1/2}\left(1+o(1)\right)G_\lambda (x,\zeta) \label{Green}}
where $G_\lambda (x,\zeta)$ denotes the Green's function defined in \eqref{Glambdadef}.

Using the transformation $U_{\zeta,\mu}(x)=\frac{1}{\varepsilon^{1/2}}V\left(\frac{x}{\varepsilon}\right)$ we see that $V$ solves the problem
\equ{\left\lbrace
\begin{array}{rcll}
-\Delta V+\varepsilon^2\lambda V-w_{\zeta',\mu'}^5&=&0&\mbox{in }\Omega_\varepsilon,\\
\ddn{V}&=&0&\mbox{on }\partial \Omega_\varepsilon,
\end{array}\right.\label{ProbAproxTrans}}
where $w_{\zeta',\mu'}(x)=3^{1/4}\frac{\mu'^{1/2}}{\sqrt{\mu'^2+|x-\zeta'|^2}}$ and $\zeta'=\varepsilon^{-1} \zeta$, $\mu'=\varepsilon^{-1}\mu$.
\newline

% % % % % % % %
The following lemma establishes the relationship between the functions $\pi_{\zeta,\mu}(x)$ and the regular part of the Green's function $G_\la (\zeta,x)$. Let us consider the (unique) radial solution $\DD_0(z)$ of the problem in entire space,
\equ{\left\lbrace
\begin{array}{rcll}
-\Delta \DD_0&=&\lambda 3^{1/4}\left[\frac{1}{\sqrt{1+|z|^2}}-\frac{1}{|z|}\right]&\mbox{in }\R^3,\\
\DD_0&\rightarrow&0&\mbox{as }|z|\rightarrow \infty.
\end{array}\right.}
$\DD_0(z)$ is a $C^{0,1}$ function with $\DD_0(z)\sim |z|^{-1}\log|z|$, as $|z|\rightarrow \infty.$
\begin{lema}{\label{lemaPi}}
For any $\sigma>0$ we have the validity of the following expansion as $\mu\rightarrow 0$
\eq{\label{expansionpi}\mu^{-1/2}\pi_{\mu,\zeta}(x)=-4\pi 3^{1/4}H_\lambda(\zeta,x)-\mu \DD_0\left(\frac{x-\zeta}{\mu}\right)+\mu^{2-\sigma}\theta(\zeta,\mu,x) .}
where for $j=0,1,2$, $i=0,1$ $i+j\leq 2$, the function $\mu^j\frac{\partial^{i+j}}{\partial \zeta^i \partial \mu^j}\theta(\zeta,\mu,x)$ is bounded uniformly on $x\in \Omega$, all small $\mu$ and $\zeta$, in compacts subsets of $\Omega$. We recall that $H_\la$ is the function defined in \eqref{robin}.
\end{lema}
\begin{proof}
Let us set $\DD_1(x)=\mu \DD_0(\mu^{-1}(x-\zeta))$, so that $\DD_1$ satisfies
\equ{\left\lbrace
\begin{array}{rcll}
-\Delta \DD_1&=&\lambda\left[\mu^{-1/2}w_{\zeta,\mu}(x)-4\pi 3^{1/4}\Gamma(x-\zeta)\right]& x\in \Omega,\\
\ddn{\DD_1}&\sim&\mu^3\log\mu& \mbox{on }\partial\Omega, \mbox{ as }\mu\rightarrow 0.
\end{array}\right.}
Let us write
$S_1(x)=\mu^{-1/2}\pi_{\zeta,\mu}(x)+4\pi 3^{1/4}H_\lambda(\zeta,x)+\DD_1(x).$
With the notation of Lemma \ref{lemaPi}, this means
\equ{S_1(x)=\mu^{2-\sigma}\theta(\mu,\zeta,x).}
Observe that for $x\in \partial \Omega$, as $\mu\rightarrow 0$,
\equ{\nabla (\mu^{-1/2}w_{\zeta,\mu}(x)+4\pi 3^{1/4}\Gamma(x-\zeta))\cdot \nu \sim \mu^2 |x-\zeta|^{-5}.}
Using the above equations we find that $S_1$ satisfies
\eq{\label{eqS1}\left\lbrace
\begin{array}{rcll}
-\Delta S_1+\lambda S_1&=&\lambda \DD_1& x\in \Omega,\\
\ddn{S_1}&=&O(\mu^3\log \mu)& \mbox{on }\partial\Omega.
\end{array}\right.}
Observe that, for any $p>3$,
\equ{\int_{\Omega}|\DD_1(x)|^pdx\leq \mu^{p+3}\int_{\R^3}|\DD_0(x)|^pdx,}
so that $\|\DD_1\|_{L^p}\leq C_p \mu^{1+3/p}$. Elliptic estimates applied to problem (\ref{eqS1}) yield that, for any $\sigma>0$, $\|S_1\|_{\infty}=O(\mu^{2-\sigma})$ uniformly on $\zeta$ in compacts subsets of $\Omega$. This yields the assertion of the lemma for $i,j=0$.

We consider now the quantity $S_2=\partial_{\zeta}S_1$. Observe that $S_2$ satisfies
\equ{\left\lbrace
\begin{array}{rcll}
-\Delta S_2+\lambda S_2&=&\lambda \partial_{\zeta}\DD_1& x\in \Omega,\\
\ddn{S_2}&=&O(\mu^3\log \mu)& \mbox{on }\partial\Omega.
\end{array}\right.}
Observe that $\partial_{\zeta}\DD_1(x)=-\nabla D_0\left(\frac{x-\zeta}{\mu}\right)$, so that for any $p>3$,
\equ{\int_\Omega |\partial_\zeta \DD_1(x)|^p dx\leq \mu^{3+p}\int_{\R^3}|\nabla \DD_0(x)|^p dx }
We conclude that $\|S_2\|_\infty=O(\mu^{2-\sigma})$, for any $\sigma>0$. This gives the proof of the lemma for $i=1$, $j=0$. Now we consider $S_3=\mu\partial_{\mu}S_1$. Then
\equ{\left\lbrace
\begin{array}{rcll}
-\Delta S_3+\lambda S_3&=&\lambda \mu \partial_{\mu}\DD_1& x\in \Omega,\\
\ddn{S_3}&=&O(\mu^3\log \mu)& \mbox{on }\partial\Omega.
\end{array}\right.}
Observe that
\equ{\mu\partial_\mu D_1(x)=\mu(\DD_0-\overline{\DD}_0)\left(\frac{x-\zeta}{\mu}\right),}
where $\overline{\DD}_0(z)=\nabla \DD_0 (z)\cdot z$. Thus, similarly as the estimate for $S_1$ itself we obtain  $\|S_3\|_\infty=O(\mu^{2-\sigma})$, for any $\sigma>0$. This yields the assertion of the lemma for $i=0$, $j=1$. The proof of the remaining estimates comes after applying again $\mu\partial_\mu$ to the equations obtained for $S_2$ and $S_3$ above, and the desired result comes after exactly the same arguments. This concludes the proof.
\end{proof}
Classical solutions to (\ref{1}) correspond to critical points of the energy functional \eqref{Energy}.
If there was a solution very close to $U_{\zeta^*,\mu^*}$ for a certain pair $(\zeta^*,\mu^*)$, then we would formally expect $E_{\lambda}$ to be nearly stationary with respect to variations of $(\zeta,\mu)$ on $U_{\zeta,\mu}$ around this point. It seems important to understand critical points of the functional $(\zeta,\mu)\rightarrow E_{\lambda}(U_{\zeta,\mu})$. In the following lemma we find explicit asymptotic expressions for this functional.
\begin{lema}
For any $\sigma>0$, as $\mu\rightarrow 0$, the following expansion holds
\eq{E_\lambda(U_{\zeta,\mu})=a_0+a_1\mu g_\lambda(\zeta)-a_2\mu^2 \lambda - a_3\mu^2g_{\lambda}^2(\zeta)+\mu^{3-\sigma}\theta(\zeta,\mu)  \label{energyExp}}
where for $j=0,1,2$, $i=0,1$, $i+j\leq 2$, the function $\mu^j\frac{\partial^{i+j}}{\partial \zeta^i \partial \mu^j}\theta(\zeta,\mu)$ is bounded uniformly on all small $\mu$ and $\zeta$ in compact subsets of $\Omega$. The $a_i$'s are explicit positive constants, given by relation \eqref{constantes} bellow.
\end{lema}

\begin{proof}
Observe that
\equ{E_{\lambda}(U_{\zeta,\mu})=\mbox{I}+\mbox{II}+\mbox{III}+\mbox{IV}+\mbox{V}+\mbox{VI},}
where
\begin{align*}
\mbox{I}	&=\int_{\Omega} \left(\frac{1}{2}|\nabla w_{\zeta,\mu}|^2-\frac{1}{6}w_{\zeta,\mu}^6  \right),\quad
\mbox{II}	=\int_{\Omega}  \left(\nabla w_{\zeta,\mu}\cdot \nabla \pi_{\zeta,\mu}-w_{\zeta,\mu}^5\pi_{\zeta,\mu} \right),\\
\mbox{III}	&=\frac{1}{2}\int_{\Omega} \left[|\nabla \pi_{\zeta,\mu}|^2+\lambda(w_{\zeta,\mu}+\pi_{\zeta,\mu})\pi_{\zeta,\mu} \right],\\
\mbox{IV}	&=\frac{\lambda}{2} \int_\Omega (w_{\zeta,\mu}+\pi_{\zeta,\mu})w_{\zeta,\mu},\quad
\mbox{V}	=-\frac{5}{2}\int_\Omega w_{\zeta,\mu}^4\pi_{\zeta,\mu}^2,\\
\mbox{VI}	&=-\frac{1}{6}\int_\Omega \left[(w_{\zeta,\mu}+\pi_{\zeta,\mu})^6-w_{\zeta,\mu}^6-6w_{\zeta,\mu}^5\pi_{\zeta,\mu}-15 w_{\zeta,\mu}^4\pi_{\zeta,\mu}^2  \right].
\end{align*}
Multiplying equation $-\Delta w_{\zeta,\mu}=w_{\zeta,\mu}^5$ by $w_{\zeta,\mu}$ and integrating by parts in $\Omega$ we obtain
\begin{align*}
\mbox{I}&=\frac{1}{2}\int_{\partial \Omega} \ddn{w_{\zeta,\mu}}w_{\zeta,\mu}+\frac{1}{3}\int_\Omega w_{\zeta,\mu}^6 \\
		&=\frac{1}{2}\int_{\partial \Omega} \ddn{w_{\zeta,\mu}}w_{\zeta,\mu}+\frac{1}{3}\int_{\R^3} w_{\zeta,\mu}^6 -\frac{1}{3}\int_{\R^3\setminus\Omega} w_{\zeta,\mu}^6.
\end{align*}
Now, testing the same equation against $\pi_{\zeta,\mu}$, we find
\equ{\mbox{II}=\int_{\partial \Omega} \ddn{w_{\zeta,\mu}}\pi_{\zeta,\mu}=-\int_{\partial\Omega}\ddn{\pi_{\zeta,\mu}}\pi_{\zeta,\mu},}
where we have used the fact that $\pi_{\zeta,\mu}$ solves problem (\ref{problempi}). Testing the equation $-\Delta \pi_{\zeta,\mu}+\lambda \pi_{\zeta,\mu}=-\lambda w_{\zeta,\mu}$ against $\pi_{\zeta,\mu}$ and integrating by parts in $\Omega$, we get
\equ{\mbox{III}=\frac{1}{2}\int_{\partial \Omega}\ddn{\pi_{\zeta,\mu}}\pi_{\zeta,\mu}.}
Testing equation $-\Delta w_{\zeta,\mu}=w_{\zeta,\mu}^5$ against $U_{\zeta,\mu}=w_{\zeta,\mu}+\pi_{\zeta,\mu}$ and integrating by parts twice, we obtain
\equ{\mbox{IV}=\frac{1}{2}\int_{\partial \Omega}\ddn{\pi_{\zeta,\mu}}\pi_{\zeta,\mu}-\frac{1}{2}\int_{\partial\Omega}\ddn{w_{\zeta,\mu}}w_{\zeta,\mu}-\frac{1}{2}\int_\Omega w_{\zeta,\mu}^5\pi_{\zeta,\mu}.}
From the mean value formula, we get
\equ{\mbox{VI}=-10\int_0^1ds(1-s)^2\int_{\Omega}(w_{\zeta,\mu}+s\pi_{\zeta,\mu})^3\pi_{\zeta,\mu}^3.}
Adding up the previous expressions we get so far
\eq{\label{expand}E_{\lambda}(U_{\zeta,\mu})=\frac{1}{3}\int_{\R^3}w_{\zeta,\mu}^6-\frac{1}{2}\int_{\Omega} w_{\zeta,\mu}^5\pi_{\zeta,\mu}-\frac{5}{2}\int_{\Omega}w_{\zeta,\mu}^4\pi_{\zeta,\mu}^2+\mathcal{R}_1,}
where
\eq{\mathcal{R}_1=-\frac{1}{3}\int_{\R^3\setminus \Omega}w_{\zeta,\mu}^6-10\int_0^1ds(1-s)^2\int_{\Omega}(w_{\zeta,\mu}+s\pi_{\zeta,\mu})^3\pi_{\zeta,\mu}^3.  \label{R1}}
We will expand the second integral term of expression (\ref{expand}). Using the change of variable $x=\zeta+\mu z$ and calling $\Omega_\mu=\mu^{-1}(\Omega-\zeta)$, we find that
\equ{A_1= \int_\Omega w_{\zeta,\mu}^5\pi_{\zeta,\mu}dx=\mu\int_{\Omega_{\mu}}w_{0,1}^5(z)\mu^{-1/2}\pi_{\zeta,\mu}(\zeta+\mu z)dz.}
From Lemma \ref{lemaPi}, we have the expansion
\equ{\mu^{-1/2}\pi_{\zeta,\mu}(\zeta+\mu z)=-4\pi 3^{1/4}H_\lambda (\zeta+\mu z,\zeta)-\mu \DD_0(z)+\mu^{2-\sigma}\theta(\zeta,\mu,\zeta+\mu z).}
According to Lemma \ref{primerlema},
\equ{H_\lambda(\zeta+\mu z,\zeta)=g_\lambda(\zeta)-\frac{\lambda}{8\pi}\mu|z|+\Theta(\zeta,\zeta+\mu z),}
where $\Theta$ is a function of class $C^2$ with $\Theta(\zeta,\zeta)=0$. Using this fact , we obtain
\equ{A_1=-4\pi 3^{1/4}\mu g_\lambda(\zeta)\int_{\R^3}w_{0,1}^5(z)dz-\mu^2\int_{\R^3}w_{0,1}^5(z)\left[\DD_0(z)-\frac{3^{1/4}}{2}\lambda |z|  \right]dz+\mathcal{R}_2   }
with
\begin{align}  \label{R2}
\mathcal{R}_2=&\mu \int_{\Omega_\mu} w_{0,1}^5(z)[\Theta(\zeta,\zeta+\mu z)+\mu^{2-\sigma}\theta(\zeta,\mu,\zeta+\mu z)  ]dz \\
&+ \mu^2\int_{\R^3\setminus \Omega_\mu}w_{0,1}^5(z)\left[\DD_0(z)-\frac{3^{1/4}}{2}\lambda |z|  \right]dz
+ 4\pi 3^{1/4}\mu g_\lambda(\zeta)\int_{\R^3\setminus \Omega_\mu}w_{0,1}^5(z)dz.  \nonumber
\end{align}
Let us recall that
$-\Delta \DD_0=3^{1/4}\lambda \left[\frac{1}{\sqrt{1+|z|^2}}-\frac{1}{|z|}\right],$
so that,
\begin{align*}
-\int_{\R^3} w_{0,1}^5\DD_0(z)&=\int_{\R^3}\Delta w_{0,1}\DD_0(z)\\
							 &=\int_{\R^3}w_{0,1}\Delta\DD_0(z)=3^{1/4}\lambda \int_{\R^3} w_{0,1}\left[\frac{1}{|z|}-\frac{1}{\sqrt{1+|z|^2}} \right].
\end{align*}
Combining the above relations we get
\begin{align*}
A_1=&-4\pi 3^{1/4}\mu g_\lambda(\zeta)\int_{\R^3}w_{0,1}^5(z)dz\\
&-\mu^2\lambda 3^{1/4}\int_{\R^3}\left[w_{0,1}(z)\left(  \frac{1}{\sqrt{1+|z|^2}}-\frac{1}{|z|} \right)-\frac{1}{2}w_{0,1}^5 |z|  \right]dz+\mathcal{R}_2.
\end{align*}
Let us consider now $A_2=\int_{\Omega}w_{\zeta,\mu}^4\pi_{\zeta,\mu}^2$. We have
\begin{align*}
A_2 =&\mu \int_{\Omega_\mu}w_{0,1}^4(z)\pi_{\zeta,\mu}^2(\zeta+\mu z)dz\\
	=&\mu^2\int_{ \Omega_\mu}w_{0,1}^4(z)\left[-4\pi 3^{1/4}H_\lambda (\zeta+\mu z,\zeta)-\mu \DD_0(z)+\mu^{2-\sigma}\theta(\zeta,\mu,\zeta+\mu z)\right]^2dz,
\end{align*}
which we expand as
\equ{A_2=\mu^2g_\lambda^2(\zeta)16\pi^23^{1/2}\int_{\R^3}w_{0,1}^4+\mathcal{R}_3.}
Combining relation (\ref{expand}) with the above expressions, we get so far
\equ{E_\lambda(U_{\zeta,\mu})=a_0+a_1\mu g_\lambda (\zeta)-a_2\lambda \mu^2-a_3\mu^2g_\lambda^2(\zeta)+\mathcal{R}_1-\frac{1}{2}\mathcal{R}_2-\frac{5}{2}\mathcal{R}_3,}
where
\begin{align*}
a_0=&\frac{1}{3}\int_{\R^3}w_{0,1}^6,\quad
a_1= 2\pi3^{1/4}\int_{\R^3}w_{0,1}^5, \quad a_3= 40\pi^23^{1/2}\int_{\R^3}w_{0,1}^4\\
a_2=&\frac{3^{1/4}}{2}\int_{\R^3} \left[w_{0,1}(z)\left( \frac{1}{|z|}-\frac{1}{\sqrt{1+|z|^2}} \right)+\frac{1}{2}w_{0,1}^5 |z|  \right]dz .
\end{align*}
An explicit computation shows that
\eq{\label{constantes} a_0=\frac{1}{4}\sqrt{3}\pi^2,\quad a_1=8\sqrt{3}\pi^2,\quad a_2=\sqrt{3}\pi^2,\quad a_3=120\sqrt{3}\pi^4.}

Finally, we want to establish the estimate
$\mu^j \frac{\partial^{i+j}}{\partial \zeta^i \partial \mu^j}\mathcal{R}_l=O(\mu^{3-\sigma}),  $
for each $j=0,1,2$, $i=0,1$, $i+j\leq 2$, $l=1,2,3$, uniformly on all small $\mu$ and $\zeta$ in compact subsets of $\Omega$. Arguing as in the proof of Lemma 2.1 in \cite{DMP} we get the validity of the previous estimates. This concludes the proof.
\end{proof}

% % % % % % % % % % % % % % %

\section{Critical single-bubbling}\label{Sec:Critical1}

The purpose of this section is to establish that in the situation of Theorem \ref{teo2} there are critical points of $E_{\lambda}(U_{\mu,\zeta})$ which persist under properly small perturbations of the functional. As we shall rigorously establish later, this analysis does provide critical points of the full functional $E_{\lambda}$, namely solutions of (\ref{1}), close to a single bubble of the form $U_{\mu, \zeta}$.
\medskip

Let us suppose the situation (a) of local maximizer:
\equ{
0 = \sup_{x\in\D} g_{\lambda_0} (x) > \sup_{x\in\partial\D} g_{\lambda_0} (x)\,.}
Then for $\lambda$ close to $\lambda_0$, $\lambda >\lambda_0$, we have
\equ{
\sup_{x\in\D} g_{\lambda} (x) > A\,(\lambda -\lambda_0),\quad A>0.}
Let us consider the shrinking set
\equ{\DD_\lambda = \left\{ y \in \D \, : \, g_\lambda (x) >  \frac A2 (\lambda -\lambda_0) \right\}.}
Assume $\lambda >\lambda_0$ is sufficiently close to $\lambda_{0}$ so that $g_\lambda = \frac A2 (\lambda -\lambda_0)$ on $\partial \DD_\lambda$.

\medskip Now, let us consider the situation of Part (b). Since $g_\lambda (\zeta )$ has a non-degenerate critical point at $\lambda = \lambda_0$ and $\zeta = \zeta_0$, this is also the case at a certain critical point $\zeta_\lambda$ for all $\lambda$ close to $\lambda_0$ where $|\zeta_\lambda -\zeta_0| = O(\lambda -\lambda_0)$.
\medskip

Besides, for some intermediate point $\tilde \zeta_\lambda$,
\equ{
g_\lambda (\zeta_\lambda ) = g_\lambda (\zeta_0 ) + Dg_\lambda ( \tilde \zeta_\lambda ) (\zeta_\lambda -\zeta_0) \ge A(\lambda- \lambda_0) + o(\lambda -\lambda_0)}
for a certain $A>0$. Let us consider the ball $B_\rho^\lambda$ with center $\zeta_\lambda$ and radius $\rho\, (\lambda -\lambda_0)$ for fixed and small $\rho >0$. Then we have that $g_\lambda (\zeta ) > \frac A2 (\lambda -\lambda_0)$ for all $\zeta \in B_\rho^\lambda$. In this situation we set $\DD_\lambda = B_\rho^\lambda$.

\medskip

It is convenient to make the following relabeling of the parameter $\mu$. Let us set
\eq{
\mu \equiv  \frac {a_1}{2\,a_2}\,\frac{g_\lambda (\zeta )}{\lambda}\, \Lambda \,,\label{mu3}}
where $\zeta \in \DD_\lambda$, and $a_1$, $a_2$ are the constants introduced in \eqref{energyExp}.  We have the following result.
%---------------------------------------------------------------------
\begin{lema}\label{psi} Assume the validity of one of the conditions (a) or (b) of Theorem \ref{teo2}, and consider a functional of the form
\eq{\psi_\lambda (\Lambda , \zeta ) = E_{\lambda } (U_{\mu , \zeta}) + g_\lambda(\zeta )^2\, \theta_\lambda (\Lambda ,\zeta )}
where $\mu$ is given by (\ref{mu3}) and
\eq{
 |\theta_\lambda | + | \nabla \theta_\lambda | + | \nabla \partial_\Lambda \theta_\lambda | \to 0, \quad {\mbox {as}} \quad \la \downarrow \la_0
\label{dla}}
uniformly on $\zeta \in \DD_\lambda$ and $ \Lambda \in (\delta , \delta^{-1})$. Then $\psi_\lambda$ has a critical point $(\Lambda_\lambda , \zeta_\lambda)$ with $\zeta_\lambda \in \DD_\lambda$, $\Lambda_\lambda \to 1$.
\end{lema}
\begin{proof} Using the expansion for the energy with $\mu$ given by (\ref{mu3}) we find now that
\eq{
\psi_\lambda ( \Lambda , \zeta ) \equiv E_{\lambda } (U_{\zeta ,\mu})+ g_\lambda(\zeta )^2\, \theta_\lambda (\Lambda ,\zeta ) = a_0+ {a_1^2\over 4\,a_2} \,\frac{g_\lambda(\zeta)^2}\lambda\,\left[ 2\,\Lambda - \Lambda^2 \right] + g_\lambda(\zeta )^2\, \theta_\lambda (\Lambda ,\zeta )}
where $\theta_\lambda$ satisfies property (\ref{dla}). Observe then that $ \partial_\Lambda \psi_\lambda = 0 $ if and only~if
\eq{
\Lambda = 1 + o(1)\,\theta_\lambda (\Lambda , \zeta) \,,}
where $\theta_\lambda$ is bounded in $C^1$-sense, as $\la \downarrow \la_0$. This implies the existence of a unique solution close to $1$ of this equation, $\Lambda = \Lambda_\lambda (\zeta ) = 1+ o(1)$ with $o(1)$ small in~$C^1$ sense, as $\la \downarrow \la_0$. Thus we get a critical point of $\psi_\lambda $ if we have one of
\eq{
p_\lambda (\zeta ) \equiv \psi_\lambda ( \Lambda_\lambda (\zeta ), \zeta ) = a_0+c\,g_\lambda (\zeta )^2\,[1 + o(1) ]}
with $o(1) \to 0$ as $\la \downarrow \la_0$ in~$C^1$-sense and $c>0$. In the case of Part (a), {\sl i.e.\/} of the maximizer, it is clear that we get a local maximum in the region $\DD_\lambda$ and therefore a critical point.

\smallskip Let us consider the case (b). With the same definition for $p_\lambda$ as above, we have
\eq{
\nabla p_\lambda (\zeta ) = 2c g_\lambda (\zeta ) \Big[\nabla g_\lambda + o(1)\, g_\lambda \Big]\,.}
Consider a point $\zeta \in \partial \DD_\lambda = \partial B_\rho^\lambda$. Then $|\nabla g_\lambda (\zeta )| = |D^2 g_\lambda (\tilde x ) (\zeta -\zeta_\lambda )| \ge \alpha \rho (\lambda -\lambda_0) $, for some $\alpha >0$, when $\la$ is close to $\la_0$. We also have $g_\lambda (\zeta) = O(\lambda-\lambda_0)$, as $\la \downarrow \la_0$. We conclude that for all $t\in (0,1)$, the function $\nabla g_\lambda + t\, o(1)\, g_\lambda $ does not have zeros on the boundary of this ball, provided that $\lambda -\lambda_0$ is small. In conclusion, its degree on the ball is constant along $t$. Since for $t=0$ is not zero, thanks to non-degeneracy of the critical point $\zeta_\lambda$, we conclude the existence of a zero of $\nabla p_\lambda (\zeta )$ inside $\DD_\lambda$. This concludes the proof.
\end{proof}
%%%%%%%%%%%%%%%%%%%%%%%%%%%%%%%%%%%%%%%%%%%%%%%%%%%%%%%%%%%%%%%%%%%%%%

\section{The linear problem}\label{sec4}
Hereafter we will look for a solution of (\ref{problem2}) of the form $v=V+\phi$, so that $\phi$ solves the problem
\eq{\left\lbrace
\begin{array}{rcll}
L(\phi)&=&N(\phi)+E&\mbox{in }\Omega_\varepsilon,\\
\ddn{\phi}&=&0&\mbox{on }\partial \Omega_\varepsilon,
\end{array}\right.\label{ProbPhi}}
where
$$
L(\phi)= -\Delta \phi +\varepsilon^2\lambda \phi-5V^4\phi,\quad
N(\phi)= (V+\phi)^5-V^5-5V^4\phi,\quad
E= V^5-w_{\zeta',\mu'}^5.
$$
Here $V$ is defined as
$U_{\zeta,\mu}(x)=\frac{1}{\varepsilon^{1/2}}V\left(\frac{x}{\varepsilon}\right)$,
where $U_{\zeta , \mu}$ is given by \eqref{ansatz1}, while $\zeta' = \varepsilon^{-1} \zeta$, and $\mu'= \varepsilon^{-1} \mu$.

Let us recall that the only bounded solutions of the linear problem
\equ{\Delta z+5w_{\zeta',\mu'}^4z=0,\quad \mbox{in }\R^3}
are given by linear combinations of the functions
$$
z_i(x)=\frac{\partial w_{\zeta',\mu'} }{\partial \zeta_i'}(x),\quad i=1,2,3, \quad
z_4(x)=\frac{\partial w_{\zeta',\mu'} }{\partial \mu'}(x).
$$
In fact, the functions $z_i,\,i=1,2,3,4$ span the space of all bounded functions of the kernel of $L$ in the case $\varepsilon=0$. Observe also that
\equ{\int_{\R^3}z_jz_k=0,\mbox{if }j\neq k.}

Rather than solving (\ref{ProbPhi}) directly, we will look for a solution of the following problem first:
Find a function $\phi$ such that for certain numbers $c_i$,
\eq{\left\lbrace
\begin{array}{rcll}
L(\phi)&=&N(\phi)+E+\sum_{i=1}^{4}c_iw_{\zeta',\mu'}^{4}z_i&\mbox{in }\Omega_\varepsilon,\\
\ddn{\phi}&=&0&\mbox{on }\partial \Omega_\varepsilon,\\
\int_{\Omega_\varepsilon}w_{\zeta',\mu'}^4z_i\phi &=&0&\mbox{for }i=1,2,3,4.
\end{array}\right.\label{ProbPhi2}}

\medskip
\noindent
We next study the linear part of the problem (\ref{ProbPhi2}). Given a function $h$, we consider the linear problem of finding $\phi$ and numbers $c_i,\,i=1,2,3,4$ such that
\eq{\left\lbrace
\begin{array}{rcll}
L(\phi)&=&h+\sum_{i=1}^{4}c_iw_{\zeta',\mu'}^4z_i&\mbox{in }\Omega_\varepsilon,\\
\ddn{\phi}&=&0&\mbox{on }\partial \Omega_\varepsilon,\\
\int_{\Omega_\varepsilon}w_{\zeta',\mu'}^4z_i\phi &=&0&\mbox{for }i=1,2,3,4.
\end{array}\right.\label{ProbPhiLin}}
Given a fixed number $0<\sigma<1$ we define the following norms
\equ{\|f\|_*\colonequals \sup_{x\in \Omega_\varepsilon} (1+|x-\zeta'|^\sigma)|f(x)|,\quad \|f\|_{**}\colonequals \sup_{x\in \Omega_\varepsilon} (1+|x-\zeta'|^{2+\sigma})|f(x)|.}
\begin{prop}{\label{PropLin}} There exist positive numbers $\delta_0$, $\varepsilon_0$, $\alpha_0$, $\beta_0$ and a constant $C>0$ such that if
\begin{equation}
\mbox{\normalfont dist}(\zeta',\partial\Omega_\varepsilon)>\frac{\delta_0}{\varepsilon}\quad \mbox{and}\quad \alpha_0<\mu'<\beta_0, \label{constra}
\end{equation}
then for any $h\in C^{0,\alpha}(\Omega_\varepsilon)$ with $\| h\|_{**}< \infty$ and for all $\varepsilon <\varepsilon_0$, problem (\ref{ProbPhiLin}) admits a unique solution $\phi=T(h)\in C^{2,\alpha}(\Omega_\varepsilon)$. Besides,
\eq{\|T(h)\|_*\leq C\|h\|_{**}\quad \mbox{and}\quad|c_i|\leq C\|h\|_{**},\, i=1,2,3,4.\label{propp}}
\end{prop}
For the proof of Proposition \ref{PropLin} we will need the next
\begin{lema}{\label{lema1}}
Assume the existence of a sequences $(\mu_n')_{n\in\N}$, $(\zeta'_n)_{n\in\N}$, $(\varepsilon_n)_{n\in\N}$ such that $\alpha_0<\mu_n'<\beta_0$, $\mbox{\normalfont dist}(\zeta'_n,\partial\Omega_\varepsilon)>{\delta_0 \over \varepsilon_n}$, $\varepsilon_n\rightarrow 0$ and for certain functions $\phi_n$ and $h_n$ with $\|h_n\|_{**}\rightarrow 0$ and scalars $c_i^n,\,i=1,2,3,4,$ one has
\equ{\left\lbrace
\begin{array}{rcll}
L(\phi_n)&=&h_n+\sum_{i=1}^{4}c_i^nw_{\zeta'_n,\mu'_n}^4z_i^n&\mbox{in }\Omega_{\varepsilon_n},\\
\ddn{\phi_n}&=&0&\mbox{on }\partial \Omega_{\varepsilon_n},\\
\int_{\Omega_{\varepsilon_n}}w_{\zeta'_n,\mu'_n}^4z_i^n\phi_n &=&0&\mbox{for }i=1,2,3,4
\end{array}\right.}
where
\equ{z_i^n=\partial_{(\zeta_n')_i}w_{\zeta_n',\mu_n'},\, i=1,2,3,\quad z_4^n=\partial_{\mu_n}w_{\zeta_n',\mu_n'} }
then
\equ{\lim_{n\rightarrow \infty}\|\phi_n\|_{*}=0}
\end{lema}
\begin{proof}By contradiction, we may assume that $\|\phi_n\|_*=1$.
We will proof first the weaker assertion that
\equ{\lim_{n\rightarrow \infty}\|\phi_n\|_\infty=0.}
Also, by contradiction, we may assume up to a subsequence that $\lim_{n\rightarrow \infty}\|\phi_n\|_\infty=\gamma$, where $0<\gamma\leq 1$. Let us see that
\equ{\lim_{n\rightarrow \infty}c_i^n=0,\,i=1,2,3,4.}
Up to subsequence, we can suppose that $\mu_n'\rightarrow \mu'$,
where $\alpha_0\leq \mu' \leq \beta_0$.
Testing the above equation against $z_j^n(x)$ and integrating by parts twice we get the relation
\equ{\int_{\Omega_{\varepsilon_n}}L(z_j^n)\phi_n+\int_{\partial\Omega_{\varepsilon_n}} \ddn{z_j^n}\phi_n=\int_{\Omega_{\varepsilon_n}}h_nz_j^n+\sum_{i=1}^{4}c_i^n\int_{\Omega_{\varepsilon_n}}w_{\zeta'_n,\mu'_n}^4z_i^nz_j^n.}
Observe that
\begin{align*}
\left| \int_{\Omega_{\varepsilon_n}}L(z_j^n)\phi_n+\int_{\partial\Omega_{\varepsilon_n}} \ddn{z_j^n}\phi_n-\int_{\Omega_{\varepsilon_n}}h_nz_j^n\right| &\leq C\|h_n\|_*+o(1)\|\phi_n\|_*, \\
\int_{\Omega_{\varepsilon_n}}w_{\zeta_n',\mu'_n}^4z_i^nz_j^n&=C\delta_{i,j}+o(1).
\end{align*}
Hence as $n\rightarrow \infty$, $c_i^n\rightarrow 0,\,i=1,2,3,4.$

Let $x_n\in\Omega_{\varepsilon_n}$ be such that $\sup_{x\in\Omega_{\varepsilon_n}}\phi_n(x)=\phi_n(x_n)$, so that $\phi_n$ maximizes at this point. We claim that there exists $R>0$ such that
\equ{|x_n-\zeta_n'|\leq R,\,\forall n\in\N.}
This fact follows immediately from the assumption $\|\phi_n\|_*=1$. We define $\tilde{\phi}_n(x)=\phi(x+\zeta'_n)$ Hence, up to subsequence, $\tilde{\phi}_n$ converges uniformly over compacts of $\R^3$ to a nontrivial bounded solution of
\equ{\left\lbrace
\begin{array}{rcll}
-\Delta \tilde{\phi}-5w_{0,\mu'}^4\tilde{\phi}&=&0&\mbox{in }\R^3,\\
\int_{\R^3}w_{0,\mu'}^4z_i\tilde{\phi} &=&0&\mbox{for }i=1,2,3,4
\end{array}\right.}
where $z_i$ is defined in terms of $\mu'$ and $\zeta'=0$.
Then $\tilde{\phi}=\sum_{i=1}^4\alpha_i z_i(x)$. From the orthogonality conditions $\int_{\R^3}w_{0,\mu'}^4z_i\tilde{\phi}=0,\,i=1,2,3,4$, we deduce that $\alpha_i=0,\,i=1,2,3,4$. This implies that $\tilde{\phi}=0$, which is a contradiction with the hypothesis $\lim_{n\rightarrow \infty}\|\phi_n\|_\infty=\gamma>0$.

\medskip

Now we prove the stronger result: $\lim_{n\rightarrow \infty}\|\phi_n\|_*=0$. Let us observe that $\zeta_n$ is a bounded sequence, so $\zeta_n\rightarrow \zeta$, as $n\rightarrow \infty$, up to subsequence. Let $R>0$ be a fixed number. Without loss of generality we can assume that $|\zeta_n-\zeta|\leq R/2$, for all $n\in \N$ and $B(\zeta,R)\subseteq \Omega$. We define $\psi_n(x)=\frac{1}{\varepsilon_n^\sigma}\phi_n\left(\frac{x}{\varepsilon_n}\right),\, x\in \Omega$ (here we suppose without loss of generality that $\mu_n>0,\,\forall n\in \N$). From the assumption $\lim_{n\rightarrow \infty}\|\phi_n\|_*=1$ we deduce that
\equ{|\psi_n(x)|\leq \frac{1}{|x-\zeta_n|^\sigma}, \mbox{ for }x\in B(\zeta,R).}
Also, $\psi_n(x)$ solves the problem
\equ{\left\lbrace
\begin{array}{rcll}
-\Delta \psi_n+\lambda \psi_n&=& \varepsilon_n^{-(2+\sigma)}\{5(\varepsilon_n^{1/2}U_{\zeta_n,\mu_n})^4\psi+g_n+\sum_{i=1}^{4}c_i^n \varepsilon_n^2w^4_{\zeta_n,\mu_n}Z_i^n\}&\mbox{in }\Omega,\\
\ddn{\psi_n}&=&0&\mbox{on }\partial \Omega, %\\
%\int_{\Omega}(1-\chi_M(|x-\zeta_n|))z_i\psi_n &=&0&\mbox{for }i=1,2,3,4
\end{array}\right.}
where $g_n(x)=h_n\left(\frac{x}{\varepsilon_n}\right)$ and $Z_i^n(x)=z_i^n\left(\frac{x}{\varepsilon_n}\right)$. Since $\lim_{n\rightarrow \infty} \|h_n\|_{**}=0$, we know that
\equ{|g_n(x)|\leq o(1)\frac{\varepsilon_n^{2+\sigma}}{\varepsilon_n^{2+\sigma}+|x-\zeta_n|^{2+\sigma}},\mbox{ for }x\in \Omega.}
Also, by (\ref{Green}), we see that
\eq{(\varepsilon_n^{1/2}U_{\zeta_n,\mu_n}(x))^4=C\varepsilon_n^4(1+o(1))G(x,\zeta_n)\label{Green2}}
%over compacts of $\overline{\Omega}\setminus\{\zeta_n\}$.
away from $\zeta_n$.
It's easy to see that $\varepsilon_n^{-\sigma}\sum_{i=1}^4c_i^n w_{\zeta_n,\mu_n}^4Z_i=o(1)$ as $\varepsilon_n\rightarrow 0$, away from $\zeta_n$.
We conclude (by a diagonal convergence method) that $\psi_n(x)$ converges uniformly over compacts of $\overline{\Omega}\setminus \{\zeta\}$ to $\psi(x)$, a bounded solution of
$$
-\Delta \psi+\lambda \psi= 0 \quad \mbox{in }\Omega\setminus \{\zeta\},\quad
\ddn{\psi}=0, \quad \mbox{on }\partial \Omega,
$$
such that $|\psi(x)|\leq \frac{1}{|x-\zeta|^\sigma}$ in $B(\zeta,R)$. So $\psi$ has a removable singularity at $\zeta$, and we conclude that $\psi(x)=0$. This implies that over compacts of $\overline{\Omega}\setminus\{\zeta\}$, we have
\equ{|\psi_n(x)|=o(1)\varepsilon_n^\sigma.}
In particular, we conclude that for all $x \in\Omega \setminus {B(\zeta_n,R/2})$ we have $|\psi_n(x)|\leq o(1)\varepsilon_n^\sigma$, which traduces into the following for $\phi_n$
\eq{|\phi_n(x)|\leq o(1)\varepsilon_n^\sigma,\mbox{ for all }x\in \Omega_{\varepsilon_n}\setminus B(\zeta_n',R/2\varepsilon_n) \label{des1}.}
Consider a fixed number $M$, such that $M< R/2\varepsilon_n$, for all $n$. Observe that $\|\phi_n\|_\infty=o(1)$, so
\eq{(1+|x|^\sigma)|\phi_n(x)|\leq o(1),\mbox{ for all }x\in \overline{B(\zeta_n',M)}\label{des2}.}
We claim that
\eq{(1+|x|^\sigma) |\phi_n(x)|\leq o(1),\mbox{ for all }x\in A_{\varepsilon_n,M},\label{des3}}
where $A_{\varepsilon_n,M}=B(\zeta'_n,R/2\varepsilon_n)\setminus \overline{B(\zeta'_n,M)}$. This assertion follows from the fact that the operator $L$ satisfies the weak maximum principle in $A_{\varepsilon_n,M}$ (choosing a larger M and a subsequence if necessary): If $u$ satisfies $L(u)\leq 0$ in $A_{\varepsilon_n,M}$ and $u\leq 0$ in $\partial A_{\varepsilon_n,M}$, then $u\leq 0$ in $A_{\varepsilon_n,M}$. This result is just a consequence of the fact that $L(|x-\zeta'_n|^{-\sigma})\geq 0$ in $A_{\varepsilon_n,M}$, if $M$ is larger enough but independent of $n$.

\medskip

We now prove (\ref{des3}) with the use of a suitable barrier. Observe that from (\ref{des1}) we deduce the existence of $\eta_n^1\rightarrow 0$, as $n\rightarrow 0$ such that $\varepsilon_n^{-\sigma}|\phi_n(x)|\leq \eta^1_n$, for all $x$ such that $|x|=R/2\varepsilon_n$. From (\ref{des2}) we deduce the existence of $\eta_n^2\rightarrow 0$, as $n\rightarrow \infty$ such that $M^\sigma|\phi_n(x)|\leq \eta_n^2$, for all $x$ such that $|x|=M$. Also, there exists $\eta_n^3\rightarrow 0$, as $n\rightarrow \infty$ such that
\equ{|x+\zeta'_n|^{2+\sigma}|L(\phi_n)|\leq \eta_n^3, \mbox{in } A_{\varepsilon_n,M}.}
We define the barrier function $\varphi_n(x)=\eta_n\frac{1}{|x-\zeta_n'|^{\sigma}}$, with $\eta_n=\max\{\eta^1_n,\eta^2_n,\eta^3_n\}$. Observe that $L(\varphi_n)=\sigma(1-\sigma)\eta_n\frac{1}{|x-\zeta_n'|^{2+\sigma}}+(\varepsilon^2_n\lambda-5V^4)\eta_n\frac{1}{|x-\zeta'_n|^{\sigma}}$. It's not hard to see that $|L(\phi_n)|\leq C L(\varphi_n)$ in $A_{\mu_n,M}$ and $|\phi_n(x)|\leq C\varphi_n$ in $\partial A_{\varepsilon_n,M}$, where $C$ is a constant independent of $n$. From the weak maximum principle we deduce (\ref{des3}) and the fact $\|\phi_n\|_{\infty}=o(1)$. From (\ref{des1}), (\ref{des2}), (\ref{des3}), and $\|\phi_n\|_{\infty}=o(1)$ we conclude that
$\|\phi_n\|_*=o(1) $
which is a contradiction with the assumption $\|\phi_n\|_*=1$. The proof of Lemma (\ref{lema1}) is completed.
\end{proof}
\begin{proof}[Proof of proposition \ref{PropLin}]
Let us consider the space
\equ{H=\left\lbrace  \phi\in H^1(\Omega)\, \vline \, \int_{\Omega_\varepsilon}w_{\zeta',\mu'}^4z_i\phi=0,\, i=1,2,3,4.  \right \rbrace}
endowed with the inner product,
$[\phi,\psi]=\int_{\Omega_\varepsilon}\nabla \phi \nabla \psi +\varepsilon^2\lambda\int_{\Omega_\varepsilon}\phi\psi.$
Problem (\ref{ProbPhiLin}) expressed in the weak form is equivalent to that of finding $\phi\in H$ such that
\equ{[\phi,\psi]=\int_{\Omega_\varepsilon}\left[5V^4\phi+h+\sum_{i=1}^4 c_i w_{\zeta',\mu'}^4z_i\right]\psi,\quad \mbox{for all }\psi\in H.}
The a priori estimate $\|T(h)\|_{*}\leq C\|h\|_{**}$ implies that for $h\equiv 0$ the only solution is $0$. With the aid of Riesz's representation theorem, this equation gets rewritten in $H$ in operational form as one in which Fredholm's alternative is applicable, and its unique solvability thus follows. Besides, it is easy to conclude (\ref{propp}) from an application of Lemma (\ref{lema1}).
\end{proof}
It is important, for later purposes, to understand the differentiability of the operator $T:h\rightarrow \phi$, with respect to the variables $\mu'$ and $\zeta'$, for a fixed $\varepsilon$ (we only let $\mu$ and $\zeta$ to vary). We have the following result
\begin{prop}{\label{prop2}}
Under the conditions of Proposition \ref{PropLin}, the map $T$ is of class $C^1$ and the derivative $\nabla_{\zeta',\mu'} \partial_{\mu'}T$ exists and is a continuous function. Besides, we have
\equ{\|\nabla_{\zeta',\mu'}T(h)\|_*+\|\nabla_{\zeta',\mu'}\partial_{\mu'} T(h)\|_*+\leq C\|h\|_{**}.}
\end{prop}
\begin{proof}
Let us consider differentiation with respect to the variable $\zeta'_k$, $k=1,2,3$. For notational simplicity we write $\frac{\partial}{\partial \zeta'_{k}}=\partial_{\zeta'_k}$. Let us set, still formally, $X_k=\partial_{\zeta'_k}\phi$. Observe that $X_k$ satisfies the following equation
\equ{L(X_k)=5\partial_{\zeta'_k}(V^4)\phi+\sum_{i=1}^4 d_i^k w_{\zeta',\mu'}^4 z_i +\sum_{i=1}^4 c_i \partial_{\zeta'_k}(w_{\zeta',\mu'}^4z_i),\quad \mbox{in }\Omega_\varepsilon.}
Here $d_i^k=\partial_{\zeta'_k}c_i$, $i=1,2,3$. Besides, from differentiating the orthogonality conditions $\int_{\Omega_{\varepsilon}} w_{\zeta',\mu'}^4z_i=0$, $i=1,2,3,4$, we further obtain the relations
\equ{\int_{\Omega_\varepsilon}X_k w_{\zeta',\mu'}^4z_i=-\int_{\Omega_\varepsilon}\phi \partial_{\zeta'_k}(w_{\zeta',\mu'}^4z_i),\quad i=1,2,3,4.}
Let us consider constants $b_i$, $i=1,2,3,4$, such that
\equ{\int_{\Omega_\varepsilon}\left(X_k-\sum_{i=1}^4 b_i z_i\right)w_{\zeta',\mu'}^4z_j=0,\quad j=1,2,3,4.}
These relations amount to
\equ{\sum_{i=1}^4 b_i\int_{\Omega_\varepsilon}w_{\zeta',\mu'}z_iz_j=\int_{\Omega_\varepsilon} \phi \partial_{\zeta'_k}(w_{\zeta',\mu'}^4z_j),\quad j=1,2,3,4.}
Since this system is diagonal dominant with uniformly bounded coefficients, we see that it is uniquely solvable and that
\equ{b_i=O(\|\phi\|_*)}
uniformly on $\zeta'$, $\mu'$ in the considered region. Also, it is not hard to see that
\equ{\|\phi \partial_{\zeta'_k}(V^4)\|_{**}\leq C\|\phi\|_*.}
From Proposition (\ref{propp}), we conclude
\equ{\left\|\sum_{i=1}^4c_i\partial_{\zeta'_k}(w_{\zeta',\mu'}^4z_i)\right\|_{**}\leq C \|h\|_{**}.}
We set $X=X_k-\sum_{i=1}^4 b_iz_i$, so $X$ satisfies
\equ{L(X)=f+\sum_{i=1}^4 b_i^k w_{\zeta',\mu'}^4 z_i,\quad \mbox{in }\Omega_\varepsilon,}
where
\equ{f=5\partial_{\zeta'_k}(V^4)\phi \sum_{i=1}^4 b_i L(z_i)+\sum_{i=1}^4c_i\partial_{\zeta',\mu'}(w_{\zeta',\mu'}^4z_i)}
Observe that also,
\equ{\int_{\Omega_\varepsilon} Xw_{\zeta',\mu'}^4z_i=0,\quad i=1,2,3,4.}
This computation is not just formal. Indeed, one gets, as arguing directly by definition shows,
$$\partial_{\xi'_k}\phi=\sum_{i=1}^4b_iz_i+T(f), \quad
{\mbox {and}} \quad
\|\partial_{\xi'_k}\phi\|_{*}\leq C\|h\|_{**}.
$$
The corresponding result for differentiation with respect to $\mu'$ follows similarly. This concludes the proof.
\end{proof}

\section{The nonlinear problem}\label{sec5}
We recall that our goal is to solve the problem (\ref{ProbPhi}). Rather than doing so directly, we shall solve first the intermediate nonlinear problem (\ref{ProbPhi2}) using the theory developed in the previous section. We have the next result

\begin{lema}{\label{lema2}}
Under the assumptions of Proposition \ref{PropLin}, there exist numbers $\varepsilon_1>0$, $C_1>0$, such that for all $\varepsilon \in (0,\varepsilon_1)$ problem (\ref{ProbPhi2}) has a unique solution $\phi$ which satisfies
\equ{\|\phi\|_*\leq C_1 \varepsilon.}
\end{lema}
\begin{proof}
First we assume that  $\mu$ and $\zeta$ are  such that $\|E\|_{**}<\varepsilon_1$.
In terms of the operator $T$ defined in Proposition (\ref{PropLin}), problem (\ref{ProbPhi2}) becomes
\equ{\phi=T(N(\phi)+E)\equiv A(\phi).}
For a given $\gamma>0$, let us consider the region
$\mathcal{F}_\gamma\colonequals \{\phi \in C(\overline{\Omega}_{\varepsilon}) \, \vline \ \|\phi\|_*\leq \gamma \|E\|_{**}  \}.$
From Proposition (\ref{PropLin}), we get
\equ{\|A(\phi)\|_*\leq C\left[ \|N(\phi)\|_{**}+\|E\|_{**}   \right].}
The definition of $N$ immediately yields $\|N(\phi)\|_{**}\leq C_0 \|\phi\|_{*}^2$. It is also easily checked that $N$ satisfies, for $\phi_1, \phi_2\in \mathcal{F}_\gamma$,
\equ{\|N(\phi_1)-N(\phi_2)\|_{**}\leq C_0\gamma \|E\|_{**}\|\phi_1-\phi_2\|_*.}
Hence for a constant $C_1$ depending on $C_0$, $C$, we get
$$
\|A(\phi)\|_*\leq C_1 \left[\gamma^2\|E\|_{**}+1 \right]\|E\|_{**} , \quad
\|A(\phi_1)-A(\phi_2)\|_*\leq C_1\gamma \|E\|_{**}\|\phi_1-\phi_2\|_{*}.
$$
Choosing
$\gamma=C_1,\quad \varepsilon_1=\frac{1}{2C_1^2},$
we conclude that $A$ is a contraction mapping of $\mathcal{F}_\gamma$, and therefore a unique fixed point of $A$ exists in this region.

\medskip
Assume now that $\mu'$ and $\zeta'$ satisfy conditions \eqref{constra}.
Recall that the error introduced by our first approximation is
\equ{
	E=V^5-w_{\zeta',\mu'}^5=(w_{\mu',\xi'}(y)+\sqrt{\varepsilon}\pi(\varepsilon y))^5-w_{\zeta',\mu'}^5(y),\quad y\in \Omega_\varepsilon.
	}
Using several times estimate \eqref{expansionpi}, we get
\equ{
\|E\|_{**}=O\left(\|\sqrt{\varepsilon}\pi(\varepsilon y)w_{\zeta',\mu'}(y)^4 \|_{**} \right)=
O\left(\left\|\varepsilon \frac{\mu'^2}{(\mu'^2+|y-\zeta'|^2)^2} \right\|_{**} \right)=O(\varepsilon),
	}
as $\varepsilon \to 0$.
This concludes the proof of the Lemma.
\end{proof}
We shall next analyze the differentiability of the map $(\zeta',\mu')\rightarrow \phi$.

We start computing the $\| \cdot \|_{**}$-norm of the partial derivatives of $E$ with respect to $\mu'$ and $\zeta'$.
Observe that
\equ{
	\partial_{\mu'}w_{\zeta',\mu'}=\frac1{2\sqrt {\mu'}} \frac{|y-\zeta'|^2-\mu'^2}{(|y-\zeta'|^2+\mu'^2)^{\frac32}}.
	}
We derive $E$ with respect to $\mu'$ and deduce
\begin{align*}
	\|\partial_{\mu'}E\|_{**}&=O\left(\|\sqrt{\varepsilon}\pi(\varepsilon y)w_{\zeta',\mu'}^3\partial_{\mu'}w_{\zeta',\mu'}\|_{**} \right)
	+O\left(\|\varepsilon^{\frac32}w^4_{\zeta',\mu'}\partial_{\mu}\pi(\varepsilon y)\|_{**} \right)\\
	&=O\left(\left\|\varepsilon \frac{\mu'(|y-\zeta'|^2-\mu'^2)}{(\mu'^2+|y-\zeta^2|)^3}\right\|_{**} \right)+O\left(\left\|\varepsilon^\frac32\frac{\mu'^2}{(\mu'^2+|y-\zeta'|^2)^2}\right\|_{**} \right)\\
	&=O(\varepsilon), \quad {\mbox {as}} \quad \varepsilon \to 0.
\end{align*}
Note that
\equ{
	|\partial_{\zeta'_i}w_{\zeta',\mu'}|= \frac{\sqrt{\mu'}|y-\zeta'|}{(\mu'^2+|y-\zeta'|^2)^\frac32},\quad \mbox{for }i=1,2,3.
	}
We derive $E$ with respect to $\zeta'_i$ and deduce for $i=1,2,3$
\begin{align*}
	\|\partial_{\zeta_i'}E\|_{**}&=O\left(\|\sqrt{\varepsilon}\pi(\varepsilon y)w_{\zeta',\mu'}^3\partial_{\zeta'_i}w_{\zeta',\mu'}\|_{**} \right)
	+O\left(\|\varepsilon^{\frac32}w^4_{\zeta',\mu'}\partial_{\zeta_i}\pi(\varepsilon y)\|_{**} \right)\\
	&=O\left(\left\|\varepsilon \frac{\mu'^2|y-\zeta'|}{(\mu'^2+|y-\zeta^2|)^3}\right\|_{**} \right)+O\left(\left\|\varepsilon^\frac32\frac{\mu'^2}{(\mu'^2+|y-\zeta'|^2)^2}\right\|_{**} \right)\\
	&=O(\varepsilon) , \quad {\mbox {as}} \quad \varepsilon \to 0.
\end{align*}
Moreover, a similar computation shows that
\equ{\|\nabla_{\zeta',\mu'}\partial_{\mu'}E\|_{**}\leq O(\varepsilon) , \quad {\mbox {as}} \quad \varepsilon \to 0.}
Collecting all the previous computations we conclude there exists a positive constant $C>0$ such that
\equ{
	\|E\|_{**}+\|\nabla_{\zeta',\mu'}E\|_{**}+\|\nabla_{\zeta',\mu'}\partial_{\mu'}E\|_{**}\leq C \varepsilon.
	}

\medskip
 Concerning the differentiability of the function $\phi(\zeta')$, let us write
\equ{A(x,\varphi)=\varphi-T(N(\varphi)+E).}
Observe that $A(\zeta',\phi)=0$ and
$\partial_\phi A(\zeta',\phi)=I+O(\varepsilon).$
It follows that for small $\varepsilon$, the linear operator $\partial_\phi A(\zeta',\phi)$ is invertible, with uniformly bounded inverse. It also depends continuously on its parameters. Differentiating respect to $\zeta'$ we obtain
\equ{\partial_{\zeta'}A(\zeta',\phi)=-(\partial_{\zeta'}T)(N(\phi)+E)-T(\partial_{\zeta'}N(\phi)+\partial_{\zeta'}R).}
where the previous expression depend continuously on their parameters. Hence the implicit function theorem yields that $\phi(\zeta')$ is a $C^1$ function. Moreover, we have
\equ{\partial_{\zeta'}\phi=-(\partial_\phi A(\zeta',\phi))^{-1}[\partial_{\zeta'}A(\zeta',\phi)].}
By Taylor expansion we conclude that
\equ{\|\partial_{\zeta'}N(\phi)\|_{**}\leq C(\|\phi\|_*+\|\partial_{\zeta'}\phi\|_*)\|\phi\|_*\leq C(\|E\|_{**}+\|\partial_{\zeta'}\phi\|_*)\|E\|_{**}.}
Using Proposition (\ref{prop2}), we have
\equ{\|\partial_{\zeta'}\phi\|_*\leq C(\|N(\phi)+E\|_{**}+\|\partial_{\zeta'}N(\phi)\|_{**}+\|\partial_{\zeta'}E\|_{**}),}
for some constant $C>0$. Hence, we conclude that
\equ{\|\partial_{\zeta'}\phi\|_{*}\leq C(\|E\|_{**}+\|\partial_{\zeta'}E\|_{**}).}
A similar argument shows that, as well
\equ{\|\partial_{\mu'}\phi\|_*\leq C(\|E\|_{**}+\|\partial_{\mu'}E\|_{**}),}
and moreover
\equ{\|\nabla_{\zeta',\mu'}\partial_{\mu'}\phi\|_*\leq C(\|E\|_{**}+\|\nabla_{\zeta',\mu'}E\|_{**}+\|\nabla_{\zeta',\mu'}\partial_{\mu'}E\|_{**}).}
This can be summarized as follows.
\begin{lema}
Under the assumptions of Proposition \ref{PropLin} and \ref{lema2} consider the map $(\zeta',\mu')\rightarrow \phi$. The partial derivatives $\nabla_{\zeta'} \phi$, $\nabla_{\mu'}\phi$, $\nabla_{\zeta',\mu'}\partial_{\mu'}$ exist and define continuous functions of $(\zeta',\mu')$. Besides, there exist a constant $C_2>0$, such that
\equ{\|\nabla_{\zeta',\mu'}\phi\|_*+\|\nabla_{\zeta',\mu'}\partial_{\mu'}\phi\|_*\leq C_2 \varepsilon}
for all $\varepsilon >0$ small enough.
\end{lema}

After Problem (\ref{ProbPhi}) has been solved, we will find solutions to the full problem (\ref{ProbPhi2}) if we manage to adjust the pair $(\zeta',\mu' )$ in such a way that $c_{i}(\zeta',\mu' )=0$, $i=1,2,3,4$. This is the {\em reduced problem}. A nice feature of this system of equations is that it turns out to be equivalent to finding critical points of a functional of the pair $(\zeta',\mu')$ which is close, in appropriate sense, to the energy of the single bubble $U$.

\section{Final argument.}\label{Sec:Variational}

In order to obtain a solution of \eqref{1} we need to solve the system of equations
\eq{ c_{j}
(\zeta',\mu' )= 0 \quad\hbox{for all } j =1,\ldots, 4\,.
\label{sist}}
If (\ref{sist}) holds, then $v= V+ \phi$ will be a solution to (\ref{ProbPhi}). This system turns out to be equivalent to a variational problem. We define
\equ{F(\zeta',\mu')=E_\varepsilon(V+\phi),}
where $\phi=\phi(\zeta',\mu')$ is the unique solution of (\ref{ProbPhi2}) that we found in the previous section, and $E_\varepsilon$ is the scaled energy functional
\equ{E_\varepsilon(U)=\frac{1}{2}\int_\Omega |\nabla U|^2+\frac{\varepsilon^2\lambda}{2}\int_\Omega |U|^2-\frac{1}{6}\int_\Omega |U|^6.}
Observe that $E_\lambda(U_{\zeta',\mu'})=E_\varepsilon(V)$.
\medskip

Critical points of $F$ correspond to solutions of \eqref{sist}, under the assumption that the error $E$ is small enough.
\begin{lema}
Under the assumptions of Propositions \ref{PropLin} and \ref{lema2}, the functional $F(\zeta',\mu')$ is of class $C^1$ and for all $\varepsilon$ sufficiently small, if $\nabla F=0$ then $(\zeta',\mu')$ satisfies system \eqref{sist}.
\end{lema}
\begin{proof}
Let us differentiate with respect to $\mu'$.
\equ{\partial_{\mu'}F(\zeta',\mu')=DE_\varepsilon(V+\phi)[\partial_{\mu'}V+\partial_{\mu'}\phi]=\sum_{j=1}^4 \int_{\Omega_\varepsilon}c_jw_{\zeta',\mu'}^4z_j[\partial_{\mu'}V+\partial_{\mu'}\phi].}
From the results of the previous section, we deduce $\partial_{\mu'}F$ is continuous. If $\partial_{\mu'}F(\zeta',\mu')=0$, then
$$
\sum_{j=1}^4 \int_{\Omega_\varepsilon}c_jw_{\zeta',\mu'}^4z_j[\partial_{\mu'}V+\partial_{\mu'}\phi]=0.
$$
Since $\|\partial_{\mu'}\|_{*}\leq C(\|E\|_{**}+\|\partial_{\mu'}E\|_{**})$, we have, as $\varepsilon\to 0$, $\partial_{\mu'}V+\partial_{\mu'}\phi=z_4+o(1)$, with $o(1)$ small in terms of the $**-$norm as $\varepsilon\to 0$. Similarly, we check that $\partial_{\zeta_k'}F$ is continuous,
\equ{\partial_{\zeta_k'}F(\zeta',\mu')=DE_\varepsilon(V+\phi)[\partial_{\zeta_k'}V+\partial_{\zeta_k'}\phi]=\sum_{j=1}^4 \int_{\Omega_\varepsilon}c_jw_{\zeta',\mu'}^4z_j[\partial_{\zeta_k'}V+\partial_{\zeta_k'}\phi]=0,}
and $\partial_{\zeta_k'}V+\partial_{\zeta_k'}\phi=z_k+o(1)$, for $k=1,2,3$.
\medskip

We conclude that if $\nabla F=0$ then
\equ{
\sum_{j=1}^4 \int_{\Omega_\varepsilon} w^4_{\zeta',\mu'}z_j[z_i+o(1)]=0,\quad i=1,2,3,4,
}
with $o(1)$ small in the sense of $**-$norm as $\varepsilon\to 0$. The above system is diagonal dominant and we thus get $c_j=0$ for all $j=1,2,3,4$.
\end{proof}
In the following Lemma we find an expansion for the functional $F$.
\begin{lema} Under the assumptions of Propositions \ref{PropLin} and \ref{lema2}, the following expansion holds
\equ{F(\zeta',\mu')=E_\varepsilon(V)+\left[\|E\|_{**}+\|\nabla_{\zeta',\mu'}E\|_{**}+\|\nabla_{\zeta',\mu'}\partial_{\mu'}E\|_{**} \right]\theta(\zeta',\mu'),}
where $\theta$ satisfies
\equ{|\theta|+|\nabla_{\zeta',\mu'}\theta|+|\nabla_{\zeta',\mu'}\partial_{\mu'}\theta|\leq C,}
for a positive constant $C$.
\end{lema}
\begin{proof}
Using the fact that $DF(V+\phi)[\phi]=0$, a Taylor expansion gives
\begin{align*}
F(V+\phi)-F(V)&=\int_0^1 D^2F(V+t\phi)[\phi,\phi](1-t)dt\\
			  &=\int_0^1 \left(\int_{\Omega_\varepsilon}\left[N(\phi)+E\right]\phi+\int_{\Omega_\varepsilon}5\left[V^4-(V+t\phi)^4\right]\phi^2 \right)(1-t)dt.
\end{align*}
Since $\|\phi\|_*\leq C\|E\|_{**}$, we get \equ{F(V+\phi)-F(V)=O(\|E\|_{**}^2).}
Observe that
\begin{align*}
\nabla_{\zeta',\mu'} &\left[ F(V+\phi)-F(V)\right] \\
&=\int_0^1 \left(\int_{\Omega_\varepsilon}\nabla_{\zeta',\mu'}\left[\left(N(\phi)+E\right)\phi\right]+\int_{\Omega_\varepsilon}5\nabla_{\zeta',\mu'}\left[\left(V^4-(V+t\phi)^4\right)\phi^2\right] \right)(1-t)dt.
\end{align*}
Since $\|\nabla_{\zeta',\mu'}\phi\|_*\leq C[\|E\|_{**}+\|\nabla_{\zeta',\mu'}E\|_{**}]$, we easily see that \equ{\nabla_{\zeta',\mu'}[F(V+\phi)-F(V)]=O(\|E\|_{**}^2+\|\nabla_{\zeta',\mu'}E\|_{**}^2).}
A similar computation yields the result.
\end{proof}
We have now all the elements to prove our main result.
\begin{proof}[Proof of Theorem \ref{teo2}]
We choose
\equ{\mu=\frac{a_1g_\lambda(\zeta)}{2a_2\lambda}\Lambda,}
where $\zeta\in \mathcal{D}_\lambda$. A similar computation to the one performed in the previous section, based in the estimate \eqref{expansionpi}, allows us to show that
\equ{
\|E\|_{**}+\|\nabla_{\zeta',\mu'}E\|_{**}+\|\nabla_{\zeta',\mu'}\partial_{\mu'}E\|_{**}\leq C\mu^{\frac12}\varepsilon^{\frac12}\delta_\la,
}
where $\delta_\lambda=\sup_{\mathcal{D}_\lambda}(|g_\la|+|\nabla g_\la|)$. Since $\alpha_0< \mu'< \beta_0$, we have
\equ{F(\zeta',\mu')=E_\varepsilon(V)+\mu^2\delta_\lambda^2\theta(\zeta',\mu'),}
with $|\theta|+|\nabla_{\zeta',\mu'}\theta|+|\nabla_{\zeta',\mu'}\partial_{\mu'}\theta|\leq C$. We define $\psi_\la(\Lambda,\zeta)=F(\zeta',\mu')$. We conclude that
\equ{\psi_\la(\Lambda,\zeta)=E_\la(U_{\zeta,\mu})+g_\la(\zeta)^2\theta_\la(\zeta,\Lambda),}
where $\theta_\la$ is as in Lemma \ref{psi}. Thus, $\psi_\la$ has a critical point as in the statement of Lemma \ref{psi}. This concludes the proof of our main result, with the constant $\gamma=\frac{a_1}{2a_2}$.
\end{proof}

\section*{acknowledgement}
The research of the first author has been suported by grants Fondecyt 1150066 and Fondo Basal CMM.  The second author has been supported by
 Millenium Nucleus CAPDE, NC130017 and Fondecyt grant 1120151.
The third author has been supported by a public grant overseen by the French National Research Agency (ANR) as part of the ``Investissements d'Avenir'' program (reference: ANR-10-LABX-0098, LabEx SMP).

\end{document}